\documentclass[final]{siamltex}
\usepackage{amsmath}
\usepackage{mathtools}
\usepackage{amssymb}
\usepackage{amsfonts}
\usepackage{amsxtra}
\usepackage{amstext}
\usepackage{amsbsy}
\usepackage{amscd}
\usepackage{graphicx}
\usepackage{float}
\usepackage{cite}
\usepackage{lmodern}
\usepackage{xcolor}
\usepackage{srcltx} 
\usepackage{marginnote,slashbox}
\usepackage{rotating}
\definecolor{darkblue}{rgb}{0.0,0.0,0.6}
\definecolor{darkgreen}{rgb}{0.0,0.6,0.0}
\usepackage[pdftex,colorlinks=true,urlcolor=darkblue,citecolor=darkblue,linkcolor=darkblue]{hyperref}
\usepackage[utf8]{inputenc}      
\usepackage{booktabs}            
\usepackage{url}                 
\usepackage[T1]{fontenc}         
\usepackage{algorithmic}         

\usepackage{calc}
\usepackage[notcite,notref]{showkeys}

\numberwithin{table}{section}    
\numberwithin{figure}{section}   
\numberwithin{equation}{section} 

\usepackage{my_latex_commands}

\newtheorem{assumption}[theorem]{Assumption}
\newtheorem{remark}[theorem]{Remark}

\newcommand{\ro}[1]{\textcolor{black}{#1}}
\renewcommand{\bl}[1]{\textcolor{black}{#1}}

\newcommand{\laz}{\lambda_0}
\newcommand{\lao}{\lambda_1}
\newcommand{\wE}{\widetilde \EE}

\newcommand{\lo}{L^2(\O)}
\newcommand{\hon}{H^1(\O)}
\newcommand{\hoon}{H^{1}(\O)^\ast}

\newcommand{\hy}{H^1_0(0,T;Y)}

\renewcommand{\o}{\omega}
\newcommand{\V}{\mathcal{V}}

\newcommand{\hhyy}{H^1_0(0,T;Y)}

\newcommand{\llun}{L^2(0,T;Y^\star)}

\newcommand{\F}{(\kappa \circ \HH)}
\newcommand{\e}{\epsilon}

\begin{document}

\title{Existence and uniqueness of solutions to \ro{rate-independent} systems with history variable}
\date{\today}
\author{Livia\ Betz\footnotemark[1]}
\renewcommand{\thefootnote}{\fnsymbol{footnote}}
\footnotetext[1]{Institut f\"ur Mathematik, Universit\"at W\"urzburg, 97074 W\"urzburg, Germany}
\renewcommand{\thefootnote}{\arabic{footnote}}
\maketitle

\begin{abstract}
This paper investigates \ro{rate-independent} systems (RIS), where the dissipation functional depends not only on the rate but also  on the \textit{history} of the state. The latter is expressed in terms of an integral operator. We  establish an existence result for the original problem and for the control thereof, without resorting to smallness assumptions. Under a  smoothness condition, we  prove the uniqueness of solutions to a certain class of \ro{history-dependent} RIS \ro{where the subdifferential of the dissipation potential is an unbounded operator}. In this context, we derive an essential estimate that opens the door to future research on the topic of optimization.\end{abstract}


\section{Introduction}\label{sec:i}

This paper is concerned with the analysis of the following \textit{\ro{history-dependent}},  \textit{\ro{rate-independent}} inclusion \begin{equation}\label{eq:n}
-\partial_q \EE(t,q(t)) \in \partial_2 \RR (\HH(q)(t),\dot{q}(t)) \quad \text{in }Y^\ast, \quad \ae (0,T),
  \end{equation}where $\EE:[0,T] \times Y \to \R$ is a \ro{smooth, uniformly convex} energy; \ro{for the precise requirements we refer to Assumption \ref{it:0}.} In \eqref{eq:n}, $T>0$ is fixed, the symbol $\partial_{2}$ denotes the  subdifferential  with respect to\ the second variable and $Y$ is a Hilbert space. 
For a Banach space $X$,  the  dissipation potential  $\RR:X \times  Y \rightarrow [0,\infty]$ is 
supposed to  be convex and positively homogeneous of degree 1  with respect to its second argument. It also satisfies other standard requirements,   a Lipschitz condition, and it has certain convergence properties with respect to both arguments, see Assumptions \ref{it:st1} and \ref{assu:r} below. 
{These 
allow for dissipation functionals arising in the modelling of \ro{history-dependent} contact problems or in the dual formulation of \ro{history-dependent} sweeping processes, see Remark \ref{rem:r}. \ro{Functionals with unbounded subdifferential} are also included in our analysis, and thus, the results apply to dissipation potentials that are typical for damage models with fatigue (Remark \ref{rem:r} and section  \ref{u0}).}
The positive homogeneity with respect to\ the second argument of $\RR,$ i.e., $\RR(\zeta,\gamma\eta)=\gamma \RR(\zeta,\eta)$ for all $\gamma \geq 0$ and for all $(\zeta,\eta)\in X \times  Y$, shows that our model is \ro{rate-independent}, i.e., the solutions to \eqref{eq:n} are not affected by time-rescaling.
The differential inclusion  \eqref{eq:n}  has the particular feature that it includes  \textit{the integral operator $\HH$}, which acts on the unknown state. This accounts for history dependence and it  is given by
\begin{equation}\label{v-op}
y \mapsto [0,T] \ni t \mapsto \HH(y)(t):=\int_0^t  {B(t-s,y(s))} \,ds +y_0 \in X,\end{equation}where  {$B:[0,T]\times Y \to X$ is a non-linear mapping satisfying certain regularity assumptions. These involve Lipschitz continuity w.r.t.\,the second variable and differentiability w.r.t.\,the time variable.}
For the precise requirements on the quantities appearing in the  definition of $\HH$ we refer to Assumption \ref{it:1} below.

The phenomenon of \textit{history dependence} in connection with differential inclusions has gained much interest in the past decades, see  e.g.\,\cite{mon} and the references therein. \ro{Rate-dependent} (viscous) problems that take the history of the state into account were tackled for instance in \cite{zv, mig_sof, sofonea, aos, jnsao}. However, when it comes to the mathematical investigation of  \textit{\ro{rate-independent} systems} (RIS) with a history component, the literature is scant. The only work known to the author dealing with this topic with regard to a rigorous analysis that involves existence of solutions is \cite{alessi}, where a fatigue damage model is investigated in two dimensions. We point out that the present paper covers the existence (and sometimes the uniqueness) of solutions to some of the  \ro{rate-independent} (non-viscous) counterparts of the aforementioned \ro{rate-dependent} problems. In particular, we show existence of solutions for  \ro{rate-independent} versions of models from \cite[Chp.\,10.4]{sofonea}, see Remark \ref{rem:r}, and unique solvability for the RIS \ro{associated with }\cite[Sec.\,4.1]{aos}, see section \ref{u0}. The case where the moving sets from \cite{zv, mig_sof} are fixed closed convex cones is treated as well, in an indirect manner, cf.\,Remark \ref{rem:r0}.

While  numerous RIS have been 
  addressed in the literature, see for instance the references in \cite{mielke}, 
   only a few deal with a dissipation potential that has \textit{two} arguments. The particularity of  such dissipation functionals is that they do not  depend only on the rate, but also on the state itself \cite{mie_ros} or they may depend on the time variable  in terms of a moving set \cite{kurzweil}. 
By contrast, the \ro{dependence}  of $\RR$ on an additional variable takes place via the  {integral} operator $\HH$, cf.\,\eqref{eq:n}. This aspect will be exploited in the analysis of the problem \eqref{eq:n} with regard to existence, uniqueness and optimal control.
As we will see, the presence of the history operator will help us prove 
satisfying results under  far \ro{fewer} assumptions than in other comparable contributions ({Remarks \ref{rem:bound} and \ref{rem:non}}).


The main results and novelties of the present  paper are:
\begin{itemize}
\item The existence of solutions for   \ro{rate-independent} models with history variable in terms of  {a (non-linear) integral} operator (Theorem \ref{ex});
\item The existence of optimal solutions for  minimization problems governed by this type of systems (Theorem \ref{ex_opt});
\item  A \textit{unique} solvability result for a special class of differential inclusions that involve  \ro{ dissipation functionals with unbounded subdifferential} (Theorem \ref{thm:lip});
\item A uniform Lipschitz estimate that will be essential in  a follow-up paper dealing with the optimal control of the aforementioned type of evolutions (Proposition \ref{lem:lip_c}).
\end{itemize}

The dependence of the dissipation potential on the state gives rise to relevant analytical difficulties both in the existence and in the uniqueness issue. 
When it comes to the existence result, the main challenge comes from the fact that we aim at  {addressing infinite-valued} dissipation functionals, where the additional dependence on the state is due to a general  {integral} operator. 
Indeed, the \ro{uniformly convex} energy in \eqref{eq:n} simplifies the analysis with regard to existence of solutions, in the sense that strong, i.e., absolutely continuous solutions are expected; note that these are often termed \textit{differential solutions} in the literature \cite[Def.\,3.3.2]{mr15}. 
In \cite{mt2004},  it was shown that such solutions exist for RIS with \textit{\ro{state-independent}} dissipation potential, provided that the energy is uniformly convex and smooth enough. RIS with convex energies and \textit{\ro{state-dependent}}  dissipation potential were investigated  in \cite{mie_ros}, where a smallness assumption is crucial for proving the existence of strong solutions. By contrast to the contribution \cite{mie_ros}, we do not need such a requirement, see Remark \ref{rem:bound} for a thorough explanation and comparison.  This is due to the fact that the \textit{state \ro{dependence}} in the first argument of $\RR$ comes from the \textit{integral operator} $\HH$.

The key result for proving existence of solutions to \eqref{eq:n} is provided by Lemma \ref{bound_k}. Its importance is twofold. First and foremost, it allows us to show uniform bounds (Proposition \ref{lem:est}) in a suitable space for  \textit{viscous} solutions, i.e, the solutions to 
\begin{equation}\label{eq:q10}
-\partial_q \EE(t,q(t)) \in \partial_2 \RR_\epsilon (\HH(q)(t),\dot{q}(t))\quad \text{in }Y^\ast  , \quad q(0) = \ro{q_0},
  \end{equation}a.e.\ in $(0,T)$, \ro{where $q_0 \in Y$}; for the precise definition of the viscous dissipation potential   $\RR_\epsilon$ \ro{associated with }the viscosity parameter $\e$ we refer to \eqref{def:r1}.
With Proposition \ref{lem:est} at hand, we can then  deduce the existence of solutions for \eqref{eq:n} (Theorem \ref{ex}) by passing to the limit $\e \searrow 0$ in \ro{the ``energy identity" associated with} \eqref{eq:q10}. Secondly, the same Lemma \ref{bound_k} enables us to establish  the boundedness of the multivalued solution map to \eqref{eq:n}, which is the essential tool for proving existence of optimal controls (Theorem \ref{ex_opt}).  {Here we maintain the discussion in a general setting. We point out that this allows for  minimization problems such as \begin{equation}\label{part}
 \left.
 \begin{aligned}
  \min \quad & \|q(T)-q_d\|^2_{Y}+\frac{1}{2}\|\ell\|^2_{H^1(0,T;Z)}\\
  \text{s.t.} \quad & (\ell,q) \in {{H^1_0(0,T;Z)}\times {H^1_0(0,T;Y)}},
  \\\quad & q \text{ solves }
\eqref{eq:n} \text{ with right hand side\ }\ell , \end{aligned}
 \quad \right\}
\end{equation}where $Z$ is a real reflexive Banach space so that $Z \embed  \embed Y^\ast$ and {$q_d \in H^1(0,T;Y)$ is a given desired state.} In applications, the  load $\ell$ plays the role of the control and it appears in the definition of the energy $\EE$. In the particular framework of \eqref{part}, the initial value of the state $q$ and of the control $\ell$ are zero, while one aims at bringing the final state as close as possible to the desired one.}

The  usual approach in the context of finding solutions for RIS with convex energies is  to resort to a time-discretization, see for instance  \cite{mt2004,mie_ros} and the references therein. However, our strategy relates to the classical \textit{\ro{vanishing-viscosity} approach} proposed by \cite{EM06}  and employed in various contributions, cf.\,e.g.\,\cite{MRS09b,  vv, vvv, MiZ09, LT11, knees, alessi, DDS11, DDS12, CL16}. All  these settings  deal with non-convex energies, such that the  \ro{vanishing-viscosity} technique is indispensable for showing the existence of so-called parametrized solutions. For other notions of weak solutions regarding RIS with non-convex energies we refer to \cite{mr15}. We underline that, despite of the convex nature of $\EE$, we also make use here of the \ro{vanishing-viscosity} approach, in view of  a follow-up paper. Therein, the aim will be to derive optimality conditions for the control of \eqref{eq:n}, by letting the viscosity parameter $\e$ tend to zero in a previously developed viscous optimality system \cite[Sec.\,4.1]{aos}, see Remark \ref{rem:fw} for more details.

Let us emphasize that one of the main novelties in the present   paper is on the topic of  \textit{uniqueness of solutions}. Here we are concerned with problems of the type
\begin{equation}\label{eq:n20}
  \begin{gathered}
-\partial_q \EE(t,q(t)) \in \partial_2 R (\HH(q)(t),\dot{q}(t)) \quad \ae (0,T), \quad q(0) = \ro{q_0},
  \end{gathered}
  \end{equation} 
where $R:\lo \times \hon \to [0,\infty]$ is given by 
 \begin{equation}\label{def:r0}
R(\zeta,\eta):=\left\{ \begin{aligned}\int_\Omega \kappa(\zeta) \,  \eta \;dx  &\quad \text{if }\eta \geq 0 \text{ a.e. in }\Omega,
\\\infty  &\quad \text{otherwise.}\end{aligned} \right.
\end{equation}
Such problems are motivated by damage models with fatigue \cite{alessi1} and our uniqueness result can be extended to other \ro{ dissipation potentials with unbounded subdifferential} as long as the linearity  with respect to their second argument is preserved, cf.\,Remark \ref{rem:lin}. 

While it is well-known  that  uniqueness of solutions  is not to be expected in RIS with non-convex energies, the situation changes when it comes to uniformly convex energies. If the dissipation functional depends only on the rate (\ro{state-independent}), uniqueness is guaranteed provided that the energy is smooth enough \cite{mt2004}. In the case of RIS with convex energy and \textit{\ro{state-dependent}} dissipation potential, uniqueness becomes again a delicate matter, as shown in \cite[Sec.\,5]{mie_ros}, where additional assumptions in terms of boundedness and smoothness of the dissipation functional are needed, see Remark \ref{rem:non} for details. Other related works are \cite{BS, TR}, where the uniqueness of (different notions of) solutions to quasi-variational sweeping processes is addressed. In \cite{BS}, the Lipschitz continuity of the gradient of the square of the Minkowski functional of the underlying set is required, as well as boundedness conditions, while in \cite{TR} additional smallness assumptions are needed to prove uniqueness.  We also mention  \cite{kurzweil}, which deals with uniqueness for \ro{rate-independent} Kurzweil processes, the dissipation functional being the Minkowski functional of a moving set. We underline that, by contrast, we do not impose bounds or smallness assumptions \ro{on the dissipation potential}. As explained  in Remark \ref{rem:non}, this is owed  to the fact that the state \ro{dependence} of $\RR$ happens through an \textit{integral operator.} In fact, our dissipation functional \ro{from \eqref{def:r0} is infinite-valued}. \ro{Besides some indispensable smoothness assumptions on the energy (Remark \ref{rem:u})}, we  only ask that the mapping $\kappa$ from \eqref{def:r0} satisfies a smoothness condition, cf.\,Assumption \ref{assu:k}. This seems to be the minimal  requirement \ro{on the dissipation potential} when it comes to proving uniqueness for quasi-variational evolutions \cite{BS, TR, mie_ros, kurzweil}, cf.\,Remark \ref{rem:non}.

The essential tool for the uniqueness result is the crucial estimate in Proposition \ref{lem:ess}. In the proof thereof, we exploit the linearity with respect to the second argument of the dissipation potential in \eqref{def:r0}, which allows  us to make use of derivation by parts formulae to obtain a suitable bound. This will enable us later  to apply Gronwall's inequality in the proofs of Theorem \ref{thm:lip} and Proposition \ref{lem:lip_c}. We refer to Remark \ref{rem:lin}  for more details. 
As an additional consequence of Proposition \ref{lem:ess}, a uniform  Lipschitz estimate  for the solution map to \eqref{eq:n20} is provided (Theorem \ref{thm:lip}). More importantly, due to Proposition \ref{lem:ess}, 
the  uniform Lipschitz continuity with respect to $\e$ of the viscous solution map can be concluded (Proposition \ref{lem:lip_c}). These findings will be valuable in a future work where  optimality conditions for the control of \eqref{eq:n20} will be established (see section \ref{sec:c} for more details). We mention here that the \ro{finite-dimensional} versions of the afore mentioned results have been recently employed in \cite{bp}. Therein, an optimality system for the control of the \ro{finite-dimensional} version of \eqref{eq:n20} was derived.




The paper is organized as follows. After introducing the notation, we state at the beginning of section \ref{2} the basic assumptions on the mappings appearing in \eqref{eq:n}. Then, we introduce the viscous \ro{system regularizing \eqref{eq:n}}, i.e., \eqref{eq:q10}, and recall its properties \cite{aos}. \ro{Here we also state some new equivalent formulations, including an energy identity (Lemma \ref{prop:ei}).} In section \ref{ub} we derive uniform bounds with respect to\ the viscosity parameter $\e$ for solutions of \eqref{eq:q10}. 
Since the state enters the first argument of the dissipation potential $\RR$ via an integral operator, we are able to prove that \ro{for all $y \in H^1(0,T;Y)$ with $\dot y (t) \in \dom \RR$  f.a.a.\, $t \in (0,T)$ it holds
\begin{equation}\label{kappa22}\begin{aligned}
\int_0^T  \Big| {\frac{d}{ds} \Big[\RR(\HH(y)(s),\dot y(t))\Big]\Big|_{s=t}}\Big|\,dt \leq c\,\|y\|^2_{C([0,T];Y)}+\frac{\alpha}{4} \|\dot y\|^2_{L^1(0,T;Y)} {+c_0},
\end{aligned}\end{equation}where  {$c,c_0>0$ depend} only on the given parameters and $\alpha>0$ is the degree of uniform convexity of $\EE$.}
Cf.\,\eqref{kappa} in Lemma \ref{bound_k}, see also Remark \ref{rem:bound}. This  is the key to showing  uniform estimates for viscous states in the desired function space  (Proposition \ref{lem:est}). In section \ref{sec:ex},
 the passage to the limit $\e \searrow 0$ in the \ro{``energy identity" associated with the }viscous differential inclusion \eqref{eq:q10} 
 allows us to prove existence of strong solutions to \eqref{eq:n} (Theorem \ref{ex}).
 \ro{An \ro{``energy identity" associated with the }rate-independent evolution is given in Proposition \ref{lem:ei}.}   Section  \ref{oc} deals with the optimization of our evolution \eqref{eq:n}, see \eqref{eq:min}. 
The estimate \eqref{kappa} is again employed, this time   to show the boundedness of the multivalued control-to-state operator (Lemma \ref{lem:est_ris}). The latter provides the starting point 
for the proof of  existence of global minimizers for  \eqref{eq:min} (Theorem \ref{ex_opt}). 
\\Section  \ref{u0} is dedicated to \ro{uniqueness results}. Here we introduce the precise assumptions on the \ro{system} \eqref{eq:n20} and show that the existence results from the previous sections apply to this setting. Subsection \ref{u} contains the main result concerning the unique solvability of \eqref{eq:n20}. 
Based on a very careful analysis, we prove a fundamental  estimate in Proposition \ref{lem:ess}. By making again use of the fact that the state enters the dissipation functional through an integral operator, we establish  this crucial result by imposing a  smoothness condition on $\kappa$ (Assumption \ref{assu:k} and  Remark \ref{rem:non}).
Proposition \ref{lem:ess} is the key to proving the Lipschitz stability of the (at first, multivalued) solution map of \eqref{eq:n20}.  As a first consequence, uniqueness follows (Theorem \ref{thm:lip}). 
A second noteworthy implication is the uniform Lipschitz continuity with respect to\ the viscosity parameter $\e$ of the respective viscous solution operator (Proposition \ref{lem:lip_c}).
The consequences of the uniqueness result and the impact on future work are discussed in section \ref{sec:c}.

\subsection*{Notation}
Throughout the paper, $T > 0$ is a fixed final time. Weak derivatives of vector-valued functions are denoted by a dot, whereas the  partial derivative with respect to the time variable of a real valued mapping $f$ is denoted by $\partial_t f$ or $\partial_s f$. If $X$ and $Y$ are Banach spaces, the notation $Y \embed \embed X$ means that $Y$ is \bl{compactly embedded} in $X$.
 For the dual pairing between $Y$ and its dual space $Y^*$
we write $\dual{.}{.}_Y$.   \bl{If $Y$ is a Hilbert space, we write $(\cdot,\cdot)_Y$ for the corresponding inner product. }
The  following abbreviations will be used throughout the paper:
\begin{equation*}
\begin{aligned}
H^{1}_0(0,T;Y)&:=\{y\in H^{1}(0,T;Y):y(0)=0\},
\\H^2(0,T;Y)&:=\{y\in H^{1}(0,T;Y):\dot y \in H^{1}(0,T;Y)\},
\end{aligned}
\end{equation*}
where $Y$ is a Banach space. 
For the polar cone of a set $ \MM \subset X$ we use the notation $\MM^\circ:=\{x^\ast \in X^\ast : \dual{x^\ast}{x}_X \leq 0 \quad \forall\, x \in \MM\}.$ 
By  $\II_{\MM}:X \to \{0,\infty\}$ we denote the indicator function \ro{associated with }the set $\MM$, i.e., $\II_\MM(y)=0$ if $y \in \MM$ and $\II_\MM(y)=\infty,$ otherwise. 
The symbol $\partial f$ stands for the convex subdifferential, see e.g.\ \cite{rockaf}.
For a mapping $\RR:X \times Y \to (-\infty,\infty]$ and a vector $x \in X$, 
the set $\partial_{2}\RR(x,y) \subset Y^\ast$ describes the convex subdifferential  at $y$ of the functional $\RR(x,\cdot):Y \to (-\infty,\infty]$. 
{For fixed $x \in X,$ the Fenchel conjugate  of the functional $\RR(x,\cdot):Y \to (-\infty,\infty]$ is given by 
$[\RR(x,\cdot)]^\star:Y^* \to [-\infty,\infty]$,  
\[[\RR(x,\cdot)]^\star(y^*):=\sup_{v \in Y} \, \dual{y^*}{v}_{Y}- \RR(x,v).\]}
\\Throughout the paper, $c,C>0$ are generic constants that depend only on the fixed physical parameters. 
To emphasize the dependence of a constant  on a certain  parameter $M$ we sometimes write $c(M).$
The Nemytskii operators associated with the mappings considered in this paper will be described by the same symbol, even when considered with different domains and ranges.
\section{Standing assumptions and the viscous problem}\label{2}

 {In this section we fix the general setting and  {comment on the required assumptions}}. \ro{Then we are concerned with some preparations for the upcoming investigations. In this context, we introduce the viscous version of \eqref{eq:n}, i.e., \eqref{eq:q1}. We recall and derive some equivalent formulations thereof  including an energy identity.} These  will be very useful later on, in the context of proving existence of solutions (sections \ref{ub} and \ref{sec:ex}).

Throughout the paper, $X$ is a Banach space, while $Y$ is a  Hilbert space.
We start the presentation of the precise framework with the assumption on the energy.
\begin{assumption}[Energy functional]\label{it:0} 
Given a  time-dependent load $\ell:[0,T] \to Y^\ast$, 
the energy $\EE:[0,T] \times Y \rightarrow \mathbb{R}$ is defined
as \ro{
\begin{equation}\label{def:e}\begin{aligned}
\EE(t, y)&:=\UU(y)-\dual{\ell(t)}{y}_{Y},
\end{aligned}\end{equation}where $\UU:Y \to \R$ is $\alpha-$uniformly convex, i.e., there exists $\alpha>0$ so that
 \begin{equation}\label{uconv}(1-\theta)\UU(y_1)+\theta \UU(y_2)\geq \UU(y_\theta)+ \frac{\alpha}{2} \theta (1-\theta)\|y_2-y_1\|_Y^2 \quad \forall\, y_1,y_2 \in Y, \forall\,\theta \in [0,1],\end{equation}where $y_\theta:=\theta y_1+(1-\theta)y_2.$ Moreover, $\UU \in C^2(Y).$}
\end{assumption}

  \ro{The uniform convexity of $\UU$ in \eqref{uconv} has two essential implications which will be used in this paper, namely:
 \begin{equation}\label{uconv1}
\UU(y_1)- \UU(y_2)\geq \UU'(y_2)(y_1-y_2)+ \frac{\alpha}{2} \|y_2-y_1\|_Y^2 \quad \forall\, y_1,y_2 \in Y,\end{equation}
\begin{equation}\label{uconv2}
\UU''(y(t))[\dot y(t)]^2 \geq  \alpha \|\dot y(t)\|_Y^2 \quad \forall\, y \in H^1(0,T;Y),\ \   \ae (0,T).\end{equation}
Note that \eqref{uconv1} follows from the definition of the derivative of $\UU$, while \eqref{uconv2} is due to a Taylor expansion of order two.}
\begin{remark}
 \ro{We remark that \eqref{uconv} is equivalent to the uniform convexity of $\EE(t,\cdot)$ for each $t\in[0,T]$. We underline that energies of the type \eqref{def:e} are often encountered in applications from continuum mechanics \cite[Sec.\,7]{mie_st_dep}. 
 From a mathematical point of view, they provide advantages, especially if one aims to optimize \eqref{eq:n} by using the applied force as the control: such energies   allow us to better follow the relation between the load and the state in terms of temporal regularity and boundedness of the control-to-state operator.
 \\
Nevertheless, the analysis in this paper up to section \ref{u0} remains unaffected if we dispense of the structure assumption \eqref{def:e} on $\EE$, i.e., if we work with a  general energy $\wE:[0,T] \times Y \rightarrow \mathbb{R}$, where the load does not explicitly appear. In addition to smoothness and the uniform convexity of $\wE(t,\cdot),$  this general energy needs to comply with other assumptions that ensure the existence of  solutions for \eqref{eq:q10}, cf.\,Remark \ref{rem:visc} below. Moreover, to be able to prove the estimate in Lemma \ref{prop:c_est}, $\wE$ has to satisfy the so-called ``uniform control of power"-condition \cite[Eq.\,(2.2)]{mie_ros}:
$\exists \laz,\lao>0$ so that
\begin{equation}\label{ucp}|\partial_t \wE(t,y)|\leq \lambda_0(\wE(t,y)+\lambda_1) \quad \forall\, t\in [0,T],\forall\,y \in Y. \end{equation}
We point out that this condition is in general not  fulfilled by our $\EE$, however we can prove Lemma \ref{prop:c_est}, thanks to the structure assumption \eqref{def:e}. We also note that to show Proposition \ref{lem:est}, one needs other additional requirements for $\wE$, such as the Lipschitz continuity of $\partial_y \wE(\cdot,y):[0,T]\to \R,$ for all $y\in Y,$ with Lipschitz constant independent of $y$.}
\end{remark}

For the integral operator appearing in the first argument of the dissipation functional we impose the following
\begin{assumption}[History operator]
\label{it:1}  The \textit{history operator} $\HH:L^1(0,T;Y) \to W^{1,1}(0,T;X)$ is given by
\begin{equation*}
 [0,T] \ni t \mapsto \HH(y)(t):=\int_0^t  {B(t-s,y(s))} \,ds +y_0 \in X,\end{equation*}
where $y_0\in X$ is fixed and  {$B:[0,T]\times Y\to X$ satisfies 
\begin{itemize}
\item for almost all $t\in (0,T),$ the mapping $B(t,\cdot):Y \to X$ is Lipschitz-continuous with Lipschitz-constant $b(t)$ satisfying $b \in L^1(0,T);$ in addition, $B(\cdot,0)\in L^1(0,T;X);$
\item for each $y \in Y$ the mapping $B(\cdot,y):[0,T] \to X$ is differentiable and $\partial_t B(\cdot,y):[0,T] \to X$ is measurable;
\item for almost all $t\in (0,T),$ the mapping $\partial_t B(t,\cdot):Y \to X$ is Lipschitz-continuous with Lipschitz-constant $\tilde b(t)$ satisfying $\tilde b \in L^1(0,T);$ in addition, $\partial_t B(\cdot,0)\in L^1(0,T;X)$.
\end{itemize}}
\end{assumption} 
We observe that, for each $y \in L^1(0,T;Y)$, the  derivative of $\HH(y)$ is given by {
\begin{equation}\label{d_h}
 [0,T] \ni t \mapsto\frac{d}{dt}  \HH(y)(t)=B(0,y(t))+\int_0^t  {\partial_t B(t-s,y(s))} \,ds .\end{equation}Moreover, we notice that
 \begin{equation}\label{d_hh}\begin{aligned}
\Big\|\frac{d}{dt}  \HH(y)(t)\Big\|_X &\leq \|B(0,y(t))\|_X+\int_0^t \|  {\partial_t B(t-s,y(s))}\|_X \,ds 
\\&\leq b(0)\|y(t)\|_Y+\|B(0,0)\|_X
\\&\quad+\int_0^t \tilde b(t-s)\|y(s)\|_Y+ \|\partial_t B(t-s,0)\|_X\,ds \quad \forall\,t\in[0,T].
\end{aligned}
\end{equation}
Thus, the  {history operator $\HH:L^1(0,T;Y) \to W^{1,1}(0,T;X)$}  is well defined, since 
\begin{align*}
\int_0^T \Big\|\frac{d}{dt}  \HH(y)(t)\Big\|_X \,dt&\leq b(0)\|y\|_{L^1(0,T;Y)}+T\,\|B(0,0)\|_X+
\\&\quad + \int_0^T \|y(t)\|_Y \int_t^T \tilde b(s-t) \,ds \,dt
\\&\quad + T  \|\partial_t B(\cdot,0)\|_{L^1(0,T;X)}
\\&\leq c(B,T)+ \|y\|_{ L^1(0,T;Y)}(b(0)+\|\tilde b \|_{L^1(0,T)}),
\end{align*}
as a consequence of Fubini's theorem and Assumption \ref{it:1}. Note that $c(B,T)>0$ is a constant that depends only on the given data. }

Finally, the dissipation potential is supposed to fulfill the following requirement.
\begin{assumption}[Dissipation functional]\label{it:st1} The dissipation functional $\RR:X \times Y  \to [0,\infty]$ has the following properties:
 \begin{enumerate}
   \item\label{it:st11} For each $\zeta \in X$, $\RR(\zeta, \cdot)$ is proper, convex, lower semicontinuous and \bl{positively} homogeneous, i.e., $\RR(\zeta,\gamma \eta)=\gamma \RR(\zeta,\eta)$ for all $\gamma \geq 0$ and all $\eta \in Y$.  The domain of $\RR(\zeta, \cdot)$ is a fixed subset of $Y$, independent of $\zeta.$ Moreover, \ro{$\RR(\zeta(\cdot), \eta(\cdot)):[0,T]\to \R$ is measurable if $\zeta:[0,T]\to X$ and $\eta:[0,T] \to \dom \RR$ are measurable.}
   \item \label{it:st12} There exists $L_\RR \geq 0$ such that 
   \begin{equation}\label{(*)}
\begin{aligned}
\RR(\zeta_1,\eta_2)-& \RR(\zeta_1,\eta_1)+\RR(\zeta_2,\eta_1)-\RR(\zeta_2,\eta_2)
\\&\qquad \qquad \qquad   \leq L_\RR\,\|\zeta_1-\zeta_2\|_X\|\eta_1-\eta_2\|_Y 
   \\& \qquad \qquad \quad \qquad  \qquad \forall\,\zeta_1,\zeta_2 \in X,\ \forall\,\eta_1,\eta_2  \in  \dom \RR,
   \end{aligned}
   \end{equation}
   where $\dom \RR$ denotes the fixed domain of $\RR(\zeta, \cdot), \ \zeta \in X$.
   \end{enumerate}\end{assumption}

\begin{remark}\label{rem:r}
We underline that Assumption \ref{it:st1} is satisfied by dissipation functionals appearing in applications. A first example is the dissipation potential $R$ from section \ref{u0}, see \eqref{def:r}. This functional appears in the modelling of damage models with fatigue \cite[Eq.\,(1.1)]{alessi} and it is straightforward to see that $R$ verifies  Assumption \ref{it:st1}.

Other possible choices include bounded  functionals
$\widehat R:X \times Y \to \R$ defined as 
 \begin{equation}\label{def:t}
\widehat R(\zeta,\eta):=(g(\zeta), | \eta|)_{X}, \end{equation}where $X$ is some Lebesgue space, while $Y$ is a subspace of $\hon$ such that $Y \embed X$, and $g:\R \to [0,\infty)$ is a Lipschitz continuous mapping. 
 
 If
\begin{equation}\label{def:y}
Y:=H^1_{\G_1}(\O)=\{y \in \hon:y=0 \text{ on }\Gamma_1\},\end{equation}where $\Gamma_1$ is a part of the boundary of a bounded Lipschitz domain $\O\subset \R^N$, $N \in \{2,3\}$, and \[X:=L^2(\Gamma_2),\] where $\Gamma_2$ is a part of the boundary of $\Omega$ such that $\G_1 \cap \G_2=\emptyset$, then the mapping $\widehat R$ defined in \eqref{def:t} is the functional appearing in the modelling of contact problems with (total) slip-dependent  friction law \cite[Eq.\,(10.65), Chp.\,10.4]{sofonea}, i.e., 
\[\widehat R(\zeta,\eta)=\int_{\G_2} g(\zeta) | \eta|\,ds.\]
 The fact that $\widehat R$ fits in the setting of Assumption \ref{it:st1} can be easily checked, see also the proof of \cite[Thm.\,10.20]{sofonea} for details.

 If $\Gamma_1$ is empty in \eqref{def:y} and  $X:=\lo,$ then
  \begin{equation}\label{def:t0}
\widehat R(\zeta,\eta)=\int_\O g(\zeta) | \eta|\,dx, \end{equation}
and  \eqref{eq:n} with $\RR:=\widehat R$ becomes the dual formulation of the \ro{history-dependent} sweeping process
 \begin{equation}\label{hds}
 \dot q(t) \in \partial \II_{K(t,q)}(-\partial_q \EE(t,q(t))) =\NN_{K(t,q)}(-\partial_q \EE(t,q(t)))\quad \ae (0,T), 
 \end{equation}
where, for a given pair $(t,q) \in [0,T] \times L^1(0,T;Y),$ one defines 
\begin{equation*}\begin{aligned}
K(t,q): =\{\phi \in L^2(\O):|\phi|\leq g (\HH(q)(t))\  \ae \O\}\end{aligned}\end{equation*} and $\NN_{K(t,q)}(-\partial_q \EE(t,q(t)))$ is the normal cone of the set $K(t,q)$ at $-\partial_q \EE(t,q(t)).$ For a detailed proof of \eqref{hds}, we refer to Appendix \ref{a} (Lemma \ref{conj0}).
We point out that the only noteworthy difference to quasivariational sweeping problems  (where $K(t,q)$ is replaced by $K(t,q(t))$) is the fact that, here, the set $K(t,q)$ depends on the history of the state $\HH(q)$, and not only on the value of the state at the present time point $t$.\end{remark}

 {\begin{remark}[Assumption \ref{it:st1}.\ref{it:st12}]\label{rem:lip}
We observe that \eqref{(*)} implies 
\begin{equation}\label{modul}
|\RR(\zeta_1,\eta)-\RR(\zeta_2,\eta)|  \leq L_\RR\,\|\zeta_1-\zeta_2\|_X\|\eta\|_Y 
\  \forall\,\zeta_1,\zeta_2 \in X,\ \forall\,\eta \in  \dom \RR,\end{equation}
which corresponds to the Lipschitz condition imposed on the finite-valued dissipation functional from  \cite{mie_ros}, cf.\,\cite[Eq.\,(3.3)]{mie_ros}.  Indeed, if we set $(\eta_1,\eta_2):=(\eta,0)$ and then $(\eta_1,\eta_2):=(0,\eta)$ in \eqref{(*)} and use the fact that $\RR(\cdot,0)=0,$ by the positive homogeneity property of $\RR$, we obtain $\eqref{(*)} \Longrightarrow \eqref{modul}$.
If $\dom \RR$ is a linear subspace of $Y,$  the opposite implication is true as well. In order to see this, we notice that
\begin{equation}\label{subl}
\RR(\zeta,\eta_1)+\RR(\zeta,\eta_2)  \leq \RR(\zeta,\eta_1+\eta_2) 
  \quad   \forall\,\zeta \in X,\ \forall\,\eta_1,\eta_2  \in Y,
   \end{equation}
which is a consequence of the convexity and $1-$homogeneity of $\RR(\zeta,\cdot),$ cf.\,Assumption \ref{it:st1}.\ref{it:st11}. Then, we have
\begin{equation*}
   \begin{aligned}
&\RR(\zeta_1,\eta_2)- \RR(\zeta_1,\eta_1)+\RR(\zeta_2,\eta_1)-\RR(\zeta_2,\eta_2)  \\&\overset{\eqref{subl}}{\leq} 
\RR(\zeta_1,\eta_2-\eta_1)+\RR(\zeta_2,\eta_1)-\RR(\zeta_2,\eta_2)
\\&\overset{\eqref{modul}}{\leq} 
L_\RR\,\|\zeta_1-\zeta_2\|_X\|\eta_1-\eta_2\|_Y +\RR(\zeta_2,\eta_2-\eta_1)+\RR(\zeta_2,\eta_1)-\RR(\zeta_2,\eta_2)
\\&\overset{\eqref{subl}}{\leq}  L_\RR\,\|\zeta_1-\zeta_2\|_X\|\eta_1-\eta_2\|_Y +\RR(\zeta_2,\eta_2)-\RR(\zeta_2,\eta_2)
\\&\quad  =L_\RR\,\|\zeta_1-\zeta_2\|_X\|\eta_1-\eta_2\|_Y
\quad \forall\,\zeta_1,\zeta_2 \in X,\ \forall\,\eta_1,\eta_2  \in  \dom \RR,
   \end{aligned}\end{equation*}
   We point out that  the second estimate in the above series of inequalities fails,  if $\eta_1,\eta_2 \in \dom \RR$ does not imply $\eta_1-\eta_2 \in \dom \RR$. To summarize, we have $\eqref{(*)} \Longleftrightarrow \eqref{modul}$, provided that $\dom \RR$ is  a linear subspace of $Y$, which shows that our Assumption \ref{it:st1}.\ref{it:st12} does not exhibit any additional restriction when compared to \cite[Eq.\,(3.3)]{mie_ros} which deals with finite-valued dissipation functionals ($\dom \RR=Y$). If  we would replace $\dom \RR$ by $Y$ in \eqref{modul} and in \eqref{(*)}, then $\eqref{(*)} \Longleftrightarrow \eqref{modul}$ would be true, but  the respective estimates then become unverifiable as terms of the type $\infty-\infty$  appear on their left-hand sides.
   \\The same statements are valid when it comes to the formulation
\begin{equation}\label{modul2}\begin{aligned}
|\RR(\zeta_1,\eta_1-\eta_2)-\RR(\zeta_2,\eta_1-\eta_2)|  &\leq L_\RR\,\|\zeta_1-\zeta_2\|_X\|\eta_1-\eta_2\|_Y 
\\&\qquad \qquad 
\  \forall\,\zeta_1,\zeta_2 \in X,\ \forall\,\eta_1,\eta_2 \in  \dom \RR.\end{aligned}\end{equation}Clearly, if $\dom \RR$ is  a linear subspace of $Y,$ then $\eqref{modul} \Leftrightarrow \eqref{modul2} \Leftrightarrow \eqref{(*)}$. Otherwise, this equivalence is not necessarily true.  By using the arguments from above, we see that \eqref{modul2} implies \eqref{(*)}. This means that we could replace the latter by (the perhaps more familiar) estimate \eqref{modul2} and the results would remain unaffected. However, for  infinite-valued dissipation functionals where the domain is not a linear subspace of $Y$, \eqref{modul2} is unverifiable as well.
 This can be seen for instance by looking at  unidirectional cases that are typical for damage models, such as  the dissipation functional from \eqref{def:r} in section \ref{u0}:  $\eta_1,\eta_2 \in H^1(\O), \eta_1,\eta_2\geq 0$ does not imply $\eta_1-\eta_2 \geq 0,$ i.e., $\eta_1-\eta_2 \in \dom R.$ Then, $\infty$ may appear on the left hand side in \eqref{modul2}, which does not happen when we work with \eqref{(*)}.
    \\On another note,  it is imperative to ask that Assumption \ref{it:st1}.\ref{it:st12} is true in order to be able to make use of the  crucial results  in Lemmas \ref{lip_F} and \ref{bound_k} below.
  \end{remark}}

\begin{remark}\label{rem:r0}
We underline that, up till now, only viscous versions of  \ro{history-dependent} sweeping processes were addressed in the literature  regarding existence of solutions \cite{zv, mig_sof} and they all deal with quadratic energies. 
If we take a closer look at \cite[Eq.\,(1.2)]{mig_sof}, we see that the history operator enters  the argument of the respective normal cone $\NN_{\mathfrak{C}(t)}$ in a linear manner, and it is not included in the description of the moving set $\mathfrak{C}(t)$, \ro{as it is the case in \eqref{hds}}. \ro{Thus, \ro{even though} it is termed ``history-dependent", the (non-viscous version) of the sweeping process addressed in \cite{mig_sof} and \eqref{hds} give rise to different mathematical challenges. In fact, the non-viscous counterpart of  \cite[Eq.\,(1.2)]{mig_sof} is  connected to (the dual formulation of) \eqref{eq:n20}.}
Indeed, if  $\mathfrak{C}(t)$  is replaced by the polar cone of the set $\CC$ from \eqref{def:c}, then the non-viscous counterpart of  \cite[Eq.\,(1.2)]{mig_sof} corresponds to
\begin{equation}\label{eq:hds0}  
\dot q(t) \in  \partial \II_{\CC^\circ}(-\partial_q \EE(t, q(t))-\F(q)(t)) \quad  \ae (0,T),\end{equation}that is, the dual formulation of  \eqref{eq:n20}. Cf.\,Lemma \ref{conj} in Appendix \ref{a}. Note that the non-linearity $\kappa$ acts on the history operator in \eqref{eq:hds0}, which is not the case in \cite[Eq.\,(1.2)]{mig_sof}. The same observations apply if we compare \eqref{eq:n20} with \cite[Eq.\,(5)]{zv}.\end{remark}

In all what follows Assumptions \ref{it:0}, \ref{it:1} and \ref{it:st1} are tacitly assumed without mentioning them every time.


%
%

\subsection*{The viscous problem}\label{sec:0}
 {Having fixed the general set-up, we can  begin our  mathematical investigation of the model \eqref{eq:n1}. We start by  presenting the viscous version thereof, which} has been recently examined in {\cite[Sec.\,3]{aos}} \ro{in the context where $\UU'$ is globally Lipschitzian}. As explained in the introduction, the viscous model will serve as the approximation of the original \ro{rate-independent} problem \eqref{eq:n} in section  \ref{sec:ex}. 
Equipped with \ro{an initial datum $q_0 \in Y$}, it reads as follows:
\begin{equation}\label{eq:q1}
-\partial_q \EE(t,q(t)) \in \partial_2 \RR_\epsilon (\HH(q)(t),\dot{q}(t))\quad \text{in }Y^\ast  , \quad q(0) =\ro{q_0},
  \end{equation}a.e.\ in $(0,T)$.
  The viscous dissipation potential   $\RR_\epsilon:X \times  Y \rightarrow [0,\infty]$ is defined as \begin{equation}\label{def:r1}
\RR_\epsilon(\zeta,\eta):=\RR(\zeta,\eta)+\frac{\epsilon}{2}\|\eta\|^2_{Y},
\end{equation}
where  $\epsilon>0$ is the viscosity parameter and $\RR$ is the 1-homogeneous  dissipation functional, cf.\,Assumption \ref{it:st1}.

%

Let us now enumerate some useful properties of the viscous evolution \eqref{eq:q1}.  {While a  unique solvability result for \eqref{eq:q1} can be found in \cite{sofonea} in the case where $\RR$ is finite-valued (not necessarily 1-homogeneous) and $\EE$ is quadratic, \cite{aos} focused on reformulating \eqref{eq:q1} as an ODE in Hilbert space in  a more general framework. This allows for dissipation functionals with unbounded subdifferential and  the derivative of $\EE$ is a globally Lipschitz function. 
Since we want to refrain from imposing additional assumptions on $\UU$ for now, we just assume for the moment that the viscous model \eqref{eq:q1} is (not necessarily uniquely) solvable. Note that the existence of (at least one) viscous solution (with suitable temporal regularity) will be enough for the upcoming analysis and that it depends on the setting under which conditions this is guaranteed. In section \ref{u0} we specify the precise conditions on $\UU$ that yield the (in that case, unique) solvability of \eqref{eq:q11}, see Assumption \ref{assu:exd} and  Remark \ref{rem:visc} below.
}
\ro{\begin{assumption}\label{assu:ex}
Let $\e>0$ be small enough. For every $\ell \in \llun$, the evolution \eqref{eq:q1} \ro{with the initial datum $q_0 \in Y$} admits at least one solution $q \in   \hy $.
\end{assumption}
\begin{remark}[Existence of viscous solutions]\label{rem:visc}
The existence of solutions for history-dependent viscous problems was  tackled in various contributions \cite{zv, mig_sof, sofonea, aos, jnsao} under the assumption that the non-linearity corresponding to our $\UU'$ is globally Lipschitzian. If we require this, we can also  deduce the unique solvability of \eqref{eq:q1}  \cite[Thm.\,8]{aos}. However, as shown at the beginning of the proof of Theorem \ref{ex1} below, the application of a fixed point argument to solve \eqref{eq:q1} is possible under weaker requirements on $\UU'$. It suffices that $\UU'$ satisfies a local Lipschitz condition and a global growth condition as in Assumption \ref{assu:exd} (where $\hon$ is replaced by $Y$). The latter ensures that the existence result holds true on the entire given time interval $[0,T].$ If we would refrain from the viscous approach, then  Assumption \ref{assu:exd}  would no longer be necessary and 
 the existence of solutions for the RIS \eqref{eq:n} could be shown by a time-discretization technique. This would involve a time-discrete energy and   the passage to the limit therein would require Assumption \ref{assu:r} below to be true \cite[Sec.\,4.1]{mie_ros}. However, this goes beyond the scope of this paper. As already pointed out in the introduction, the starting point for our investigations is the viscous model \eqref{eq:q1}, in view of a future work that involves satisfying optimality conditions associated to the control of \eqref{eq:n}, see \cite{aos} and Remark \ref{rem:fw}.\end{remark}}
\begin{lemma}[The viscous problem as a non-smooth ODE, {\cite[Lemma 2]{aos}}]\label{lem:proj}
Let $\ell \in \llun$. Then, $q\in \hy$ satisfies the evolution in \eqref{eq:q1} if and only if 
\begin{equation}\label{eq:syst_diff1}
 \dot q(t)   =\frac{1}{\e} \V_Y^{-1}(\mathbb{I}-P_{\partial_2 \RR(\HH(q)(t),0)})(-\partial_q \EE(t,q(t)) ) \quad \text{in }Y, \quad \ae (0,T),
 \end{equation}where $\V_Y:Y \to Y^\ast$ is the Riesz isomorphism \ro{associated with }the Hilbert space $Y$. For each $\zeta \in X$, the operator $P_{\partial_2 \RR(\zeta, 0)}:Y^\ast \to Y^\ast $ is the \bl{(metric)} projection onto the set \[\partial_2 \RR(\zeta, 0)=\{\varphi \in Y^\ast: \dual{\varphi}{v}_Y \leq \RR(\zeta,v) \quad \forall\,v \in Y\}\] with respect to\ the \bl{inner product $\dual{\V_Y^{-1}\cdot }{\cdot }_{Y^\ast}$}, i.e.,  $P_{\partial_2 \RR(\zeta, 0)}\o$ is the unique solution of \begin{equation*} \min_{\mu \in \partial_2 \RR(\zeta, 0)} \frac{1}{2}\|\o-\mu\|_{\V_Y^{-1}}^2
\end{equation*} 
for any $\o \in Y^\ast$.
\end{lemma}

 {In \cite{aos}, an essential Lipschitz continuity property was established. This is based on Assumption \ref{it:st1}.\ref{it:st12} and it will be useful in the proof of Proposition \ref{lem:est} below. \ro{The next result is also a key tool when it comes to solving the ODE \eqref{eq:syst_diff1}, see the proof of Theorem \ref{ex1} below for details.}}

\begin{lemma}[Lipschitz continuity of the projection, {\cite[Lemma 3]{aos}}]\label{lip_F}
The function \[X \times Y^\ast \ni (\zeta,\omega) \mapsto \frac{1}{\e} \V_Y^{-1}(\mathbb{I}-P_{\partial_2 \RR(\zeta,0)})\o\in Y\] is globally Lipschitz continuous.\end{lemma} 

 {In section \ref{ub}, we intend to show uniform bounds w.r.t.\,$\e$ for viscous solutions associated to \eqref{eq:q1}, see Proposition \ref{lem:est} below. For this, we need a suitable reformulation of \eqref{eq:q1}, which is given by the following}

\begin{lemma}[Another equivalent formulation]\label{prop:equiv}
Let $\ell \in \llun$. Then, each solution \ro{$q \in   \hy $} \ro{associated with }\eqref{eq:q1} satisfies 
 \begin{equation}\label{eq:ic}  
-\partial_q \EE(t,q(t))-\e\V_Y \dot q(t) \in \partial_2 \RR(\HH(q)(t),\dot q(t)) \quad  \text{ in }Y^\ast,\  \ae (0,T) .   \end{equation}In particular,  \begin{subequations}  \begin{gather}
\dual{  -\partial_q \EE(t,q(t))-\e\V_Y \dot q(t) }{\dot q(t)}_{Y}=\RR(\HH(q)(t),\dot q(t)),\label{eq:p1}
\\  \dual{  -\partial_q \EE(t,q(t))-\e \V_Y \dot q(t) }{v}_{Y}\leq \RR(\HH(q)(t),v), \quad \forall\, v \in Y,  \quad   \ae (0,T), \label{eq:polar}\end{gather}\end{subequations}where $\V_Y:Y \to Y^\ast$ is the Riesz isomorphism \ro{associated with }the Hilbert space $Y$.
\end{lemma}
\begin{proof}
The first assertion is due to the sum rule for convex subdifferentials, while \eqref{eq:p1}-\eqref{eq:polar} is an immediate consequence of the positive homogeneity of $\RR(\HH(q)(t),\cdot),$ cf.\,Assumption \ref{it:st1}.\ref{it:st11}.\end{proof}

\ro{In section \ref{sec:ex},  the solvability of \eqref{eq:n} equipped with initial datum $q_0 \in Y$ is proven by passing to the limit $\e \searrow 0$ in the viscous model. Here, the reformulation of \eqref{eq:q1} in terms of an ``energy identity" \cite{mielke} will turn out to be suitable. The starting point therefor is a chain rule for the total time-derivative of the energy, which is due to the smoothness of $\UU$ and the embedding $H^{1}(0,T;Y^\ast) \embed L^{\infty}(0,T;Y^\ast)$. This is contained in the following \begin{lemma}\label{lem:chain}
If $\ell \in H^{1}(0,T;Y^\ast)$ and $y \in H^{1}(0,T;Y),$ then \[\EE(\cdot,y(\cdot)) \in H^{1}(0,T;\R)\] with 
\begin{equation}\label{chain}
\frac{d}{dt} \EE(t,y(t))=\partial_t \EE(t,y(t))+\dual{\partial_y \EE(t,y(t))}{\dot y(t)}_{Y} \quad \ae (0,T).
\end{equation}Moreover,
\begin{equation}\label{deriv}\begin{aligned}
\partial_t \EE(t,y)&=-\dual{\dot \ell(t)}{y}_{Y} \quad \ae (0,T),\forall\,y\in Y,
\\\partial_y \EE(t,y)&= \UU'( y) -\ell(t) \quad \forall\,(t,y)\in [0,T]\times Y.
\end{aligned}\end{equation}
\end{lemma}
To prove Theorem \ref{ex} below, as well as a first uniform bound for the viscous solution (Lemma \ref{prop:c_est}), we will use the following
\begin{lemma}[Energy identity in the viscous case]
\label{prop:ei}
Let $\ell \in W^{1,1}(0,T;Y^\ast)$. Then, each solution $q$ of the viscous differential inclusion \eqref{eq:q1} satisfies
\begin{equation}\label{eq:ei}\begin{aligned}
\int_0^t  \RR(\HH(q)(\tau),\dot q(\tau)) d \tau +\int_0^t &\e \| \dot q (\tau)\|^2_{Y} d \tau + \EE(t,q(t))
\\& =\EE(0,q(0))+  \int_0^t  \partial_t \EE(\tau,q(\tau))  d \tau \quad \forall\,0 \leq t\leq T.\end{aligned}
\end{equation}
 \end{lemma}
\begin{proof}The result follows by integrating  \eqref{eq:p1} over $[0,t]$ and applying  \eqref{chain}.\end{proof}
}

\section{Uniform bounds}\label{ub}
This section is dedicated to proving that the viscous solutions $q_\e$ of \eqref{eq:q1} are uniformly bounded with respect to\ the viscosity parameter $\e$ in a suitable space. This will allow us later on in {section \ref{sec:ex}}  {to extract weakly convergent subsequences, so that a limit analysis of  the viscous problem as $\e$ approaches zero may be performed}.  {We aim at proving uniform bounds in $H^1(0,T;Y)$, by first showing a weaker estimate, which is contained in the next lemma. Then, we state an essential result (Lemma \ref{bound_k}) that in combination with 
Lemma \ref{prop:c_est} below will yield the desired uniform estimate for $q_\e$ (Proposition \ref{lem:est}). }
\begin{lemma}\label{prop:c_est}
Let $\ell \in W^{1,1}(0,T;Y^\ast)$ and let \ro{$q_\e \in   \hy $ be a solution to }\eqref{eq:q1} with right hand side $\ell$.
Then, it holds \begin{equation}\label{eq:c}
\|q_\e\|_{C([0,T];Y)} \leq c\, \|\ell\|_{W^{1,1}(0,T;Y^\ast)}\ro{+c_0},
 \end{equation}where $c,c_0>0$ are constants independent of $\e.$
   \end{lemma}  
   \begin{proof}
\ro{We start the proof by recalling the result from Lemma \ref{prop:ei}, according to which one has
\begin{equation}\label{eq:eiq}\begin{aligned}
\int_0^t  \RR (\HH(q_\e)(\tau),\dot q_\e(\tau)) d \tau +\int_0^t &\e \| \dot q_\e (\tau)\|^2_{Y} d \tau + \EE(t,q_\e(t))
\\& =\EE(0,q_0)+  \int_0^t  \partial_t \EE(\tau,q_\e(\tau))  d \tau \quad \forall\,0 \leq t\leq T .\end{aligned}
\end{equation}
Since the dissipation functional takes only nonnegative values, cf.\,Assumption \ref{it:st1}, one obtains  
\begin{equation*}\begin{aligned}
\EE(t,q_\e(t))
\leq \EE(0,q_0)+  \int_0^t  \partial_t \EE(\tau,q_\e(\tau))  d \tau \quad \forall\,0 \leq t\leq T.
\end{aligned}
\end{equation*}By using \eqref{def:e} and \eqref{deriv}, one deduces 
\begin{equation*}\begin{aligned}
\UU(q_\e(t))-\UU(q_0) \leq \dual{\ell(t)}{ q_\e(t)}_{Y}-\dual{\ell(0)}{ q_0}_{Y}-\int_0^t \dual{\dot \ell(\tau)}{ q_\e(\tau)}_{Y}  d \tau \quad \forall\,0 \leq t\leq T.
\end{aligned}
\end{equation*}In view of \eqref{uconv1}, it then follows 
\begin{equation}\label{aest}\begin{aligned}
\UU'(q_0)(q_\e(t)-q_0)+&\frac{\alpha}{2}\|q_\e(t)-q_0\|_Y^2  \leq \|\ell\|_{C([0,T];Y^\ast)}\|q_\e\|_{C([0,T];Y)}\\&\qquad + \|\dot \ell\|_{L^1(0,T;Y^\ast)}\|q_\e\|_{C([0,T];Y)}\ro{+\widehat c_0} \quad \forall\,0 \leq t\leq T,
\end{aligned}
\end{equation}where $\widehat c_0>0$ depends only on $q_0$ and $\ell(0).$ Since $\UU$ is convex and continuous, by Assumption \ref{it:0}, $\UU$ is locally Lipschitz continuous \cite[Prop.\,2.107]{b_s}, i.e, there exists $L_{q_0}>0$ so that 
\[|\UU'(q_0)(q_\e(t)-q_0)| \leq L_{q_0}\|q_\e(t)-q_0\|_Y \quad \forall\,0 \leq t\leq T.\]
Using this in \eqref{aest} results in 
\begin{equation*}\begin{aligned}
\frac{\alpha}{2}\|q_\e-q_0\|_{C([0,T];Y)}^2 &\leq \|\ell\|_{C([0,T];Y^\ast)}\|q_\e\|_{C([0,T];Y)}\\&\qquad + \|\dot \ell\|_{L^1(0,T;Y^\ast)}\|q_\e\|_{C([0,T];Y)}+L_{q_0}\|q_\e-q_0\|_{C([0,T];Y)}\ro{+\widehat c_0},\end{aligned}
\end{equation*}
which in light of Young's inequality yields
\begin{equation*}\begin{aligned}
\frac{\alpha}{4}\|q_\e-q_0\|_{C([0,T];Y)}^2 &\leq \|\ell\|_{C([0,T];Y^\ast)}\|q_\e\|_{C([0,T];Y)}\\&\qquad + \|\dot \ell\|_{L^1(0,T;Y^\ast)}\|q_\e\|_{C([0,T];Y)}\ro{+\widetilde c_0},\end{aligned}
\end{equation*}where $\widetilde c_0>0$ depends only on $\alpha, \widehat c_0$  and $L_{q_0}.$
Since \[\frac{\alpha}{8}\|q_\e\|_{C([0,T];Y)}^2\leq \frac{\alpha}{4}\|q_\e-q_0\|_{C([0,T];Y)}^2+\frac{\alpha}{4}\|q_0\|_{Y}^2,\]the previous estimate leads to 
\begin{equation*}\begin{aligned}
\frac{\alpha}{8}\|q_\e\|_{C([0,T];Y)}^2 &\leq \|\ell\|_{C([0,T];Y^\ast)}\|q_\e\|_{C([0,T];Y)}\\&\qquad + \|\dot \ell\|_{L^1(0,T;Y^\ast)}\|q_\e\|_{C([0,T];Y)}\ro{+\widetilde c_0}+\frac{\alpha}{4}\|q_0\|_{Y}^2.\end{aligned}
\end{equation*}
By applying again Young's inequality we can now conclude that there exist constants \ro{$c,c_0>0$, independent of $\e$,} so that
 \eqref{eq:c} is satisfied.}
 \end{proof}
 {As advertised at the beginning of this section, the next lemma will allow us to improve the estimate from Lemma \ref{prop:c_est}, thanks to Assumption \ref{it:st1}.\ref{it:st12} and to the fact that the state-dependence of $\RR$ is due to the  history operator $\HH$, cf.\,Assumption \ref{it:1}.}
\begin{lemma}\label{bound_k}
Let $y \in H^1(0,T;Y)$ be given with $\dot y (t) \in \dom \RR$  f.a.a.\, $t \in (0,T)$. Then, f.a.a.\, $t \in (0,T),$ the mapping 
 \[ [0,T] \ni s \mapsto \RR(\HH(y)(s),\dot y(t)) \in \R\]
 is Lipschitz continuous on $[0,T]$. Moreover, it holds
  \begin{equation}\label{kappa}\begin{aligned}
\int_0^T  \Big| {\frac{d}{ds} \Big[\RR(\HH(y)(s),\dot y(t))\Big]\Big|_{s=t}}\Big|\,dt &\leq 
L_\RR \int_0^T \Big\|\frac{d}{dt}\HH(y)(t)\Big\|_X\|\dot y(t)\|_Y\,dt 
\\&\leq c\,\|y\|^2_{C([0,T];Y)}+\frac{\alpha}{4} \|\dot y\|^2_{L^1(0,T;Y)} {+c_0},
\end{aligned}\end{equation}where  {$c,c_0>0$ depend only on $L_\RR, b,\tilde b, B$ and $\alpha$.}
\end{lemma}

\begin{proof}
Since $\dot y (t) \in \dom \RR$  f.a.a.\,$t \in (0,T)$, we can make use of Assumption \ref{it:st1}.\ref{it:st12} to deduce the  Lipschitz continuity of the mapping \[X \ni \zeta \mapsto \RR(\zeta,\dot y(t)) \in \R, \quad \text{f.a.a.\,} t \in (0,T).\] Moreover, $\HH(y) \in W^{1,\infty}(0,T;X),$ due to $y \in H^1(0,T;Y)\embed L^\infty(0,T;Y)$ and \eqref{d_hh}. Hence, for a.a.\,$t \in [0,T]$,  $\RR(\HH(y)(\cdot),\dot y(t))$ is a composition of two Lipschitz continuous functions, which proves the first assertion.

To show the estimate \eqref{kappa}, we first observe that, f.a.a.\, $t \in (0,T),$ it holds 
  \begin{equation}\label{d_R}
   \Big| {\frac{d}{ds} \Big[\RR(\HH(y)(s),\dot y(t))\Big]}\Big|\leq  L_\RR \Big\|\frac{d}{ds}\HH(y)(s)\Big\|_X\|\dot y(t)\|_Y \quad \text{f.a.a.\,}s \in (0,T),
   \end{equation}where we employed again Assumption \ref{it:st1}.\ref{it:st12}.
 In view of \eqref{d_hh}, we further have 
 {\begin{equation}\label{estt}
\begin{aligned}
& \int_0^T \Big\|\frac{d}{dt}\HH(y)(t)\Big\|_X\|\dot y(t)\|_Y\,dt 
\\&\leq \int_0^T \Big(  { b(0)\|y(t)\|_Y+\|B(0,0)\|_X+ \int_0^t \|\partial_t B(t-s,0)\|_X\,ds} \Big)\|\dot y(t)\|_Y\,dt 
\\&\quad +\int_0^T  {\int_0^t \tilde b(t-s)\|y(s)\|_Y\,ds} \|\dot y(t)\|_Y\,dt 
\\&\leq b(0) \|y\|_{C([0,T];Y)} \|\dot y\|_{L^1(0,T;Y)}+c(B)
\|\dot y\|_{L^1(0,T;Y)}
\\&\quad +\|\tilde b\|_{L^1(0,T)} \|y\|_{C([0,T];Y)}\|\dot y\|_{L^1(0,T;Y)}
\\&\leq C(b,\tilde b,\alpha,L_{\RR})\,\|y\|^2_{C([0,T];Y)}+\frac{\alpha}{4L_{\RR}} \|\dot y\|^2_{L^1(0,T;Y)}+\tilde c(B,\alpha,L_{\RR}),
\end{aligned}
  \end{equation}where $C(b,\tilde b,\alpha,L_{\RR}),\tilde c(B,\alpha,L_{\RR})>0$ are constants that depend only on the given data;} note that in the last estimate we applied Young's inequality. In the light of \eqref{estt}, setting $s:=t$ in \eqref{d_R} and integrating the resulting estimate over time finally implies \eqref{kappa}. 
  \end{proof}
\begin{remark}\label{rem:bound}
  Lemma \ref{bound_k} will play an essential role not only in the proof of Proposition \ref{lem:est} below (which is crucial for showing existence of strong solutions to \eqref{eq:n}), but also later on, in section \ref{oc}, where the existence of global optimizers of \eqref{eq:min} is established (Theorem \ref{ex_opt}).
  \\
We underline that the proof of Lemma \ref{bound_k} exploits the fact that the history variable appearing in the first argument of $\RR$ is expressed in terms of an integral operator, cf.\,Assumption \ref{it:1}. Thanks to this aspect, we do not need to require smallness conditions  on the involved quantities to ensure the existence of solutions to \eqref{eq:n}, as opposed to other comparable contributions \cite{mie_ros, BS}. 
\\
In \cite{mie_ros} for instance,  the necessity of such an assumption is proven, as the state \ro{dependence} of the dissipation potential therein is due to a Nemytskii operator. 
We point out that this is consistent with our estimate \eqref{kappa}. Indeed, if we consider the simple case where $\HH$ is replaced by the  identity operator $Y \embed X$, then 
  \begin{equation*}
\int_0^T \Big\|\frac{d}{dt}\HH(y)(t)\Big\|_X\|\dot y(t)\|_Y\,dt  \leq  \|\dot y\|^2_{L^2(0,T;Y)}.
\end{equation*}Thus, one can no longer make use of \eqref{eq:c} in the proof of \eqref{eq:est1} below, since \eqref{eq:syst_dif222} is no longer available.
Instead of \eqref{eq:syst_dif222} below, one has  the estimate
 \begin{equation*}
\begin{aligned}
\alpha  \ro{\| \dot q_\e\|_{L^2(0,T;Y)}^2}& \leq  \|\ell\|_{H^1(0,T;Y^\ast)}    \ro{\| \dot q_\e\|_{L^2(0,T;Y)}}  \\&\quad+L_{\RR} \|\dot q_\e\|^2_{L^2(0,T;Y)},
  \end{aligned}\end{equation*}cf.\,the first inequality in \eqref{kappa}. 
Hence, the proof of \eqref{eq:est1} below is successful in the case $\HH=\mathbb{I}$ if 
\[L_\RR < \alpha,\]which is precisely the condition  \cite[Eq.\,(1.2)]{mie_ros}; note that our \ro{uniformly convex energy $\EE$} fits in the setting from \cite{mie_ros} if the parameter called $\kappa$ in \cite{mie_ros} takes the value $\alpha.$  {Regarding the above smallness assumption, we also refer to \cite[Sec.\,4.2]{mie_st_dep}, where a comparison to \cite{BS} is made.}
\end{remark}

In order to be able to prove uniform bounds for $q_\e$ in \ro{$\hy$}, and later on,  existence of solutions for \eqref{eq:n}, we need the following
\begin{definition}[Compatibility condition]\label{def}
In all what follows, we abbreviate
\[\mathfrak{L}(\ro{q_0,y_0}):=\{\ell \in C([0,T];Y^\ast):\ro{-\UU'(q_0)+}\ell(0) \in \partial_2 \RR(y_0,0) \},\]
where \ro{$q_0\in Y$ is the initial datum in \eqref{eq:q1}} and $y_0 \in X$ is given by Assumption \ref{it:1}.
\end{definition}
\begin{remark}
In view of our goal to show existence of differential solutions to \eqref{eq:n}, 
the necessity of the compatibility condition  in Definition \ref{def} should not be surprising. This type of requirement ensures that \eqref{eq:n} is true at the initial time point when the rate of the state   is zero, i.e., {$-\partial_q \EE(0,q(0)) \in \partial_2 \RR (\HH(q)(0),0)$}, see for instance \cite[Thm.\,3.5.2,\,p.\,162]{mr15}.
\end{remark}
 {The main result of this section is contained in the following}
\begin{proposition}[Uniform boundedness of  viscous solutions]\label{lem:est}
\ro{Let Assumption \ref{assu:ex} hold and let $\ell \in H^{1}(0,T;Y^\ast)$. Then, each solution of the viscous  problem \eqref{eq:q1} satisfies $q_\e \in   H^2(0,T;Y) $}. If $\ell \in H^{1}(0,T;Y^\ast) \cap \mathfrak{L}(\ro{q_0,y_0})$, then it holds \begin{equation}\label{eq:est1}
\|q_\e\|_{\ro{\hy}} \leq C\,\|\ell\|_{H^1(0,T;Y^\ast)}\ro{+c},
 \end{equation}where $C,c>0$ are independent of $\e.$
   \end{proposition} 
   \begin{proof}\ro{By Assumption \ref{assu:ex},  \eqref{eq:q1} admits at least one  solution \ro{$q_\e \in   \hy $}.}
 Thus, \ro{in view of $\ell \in H^1(0,T;Y^\ast)$, \eqref{deriv} and $\UU'\in C^1(Y)$, cf.\,Assumption \ref{it:0}, we have the regularity 
 \[\partial_q  \EE(\cdot, q_\e(\cdot))=\UU'(q_\e)-\ell \in H^1(0,T;Y^\ast).\]} By  the Lipschitz continuity of \[X \times Y^\ast \ni (\zeta,\omega) \mapsto \frac{1}{\e} \V_Y^{-1}(\mathbb{I}-P_{\partial_2 \RR(\zeta,0)})\o\in Y,\]cf.\,Lemma \ref{lip_F}, and since $q_\e$ satisfies \eqref{eq:syst_diff1} (Lemma \ref{lem:proj}), we then deduce that \[q_\e \in H^2(0,T;Y),\] see e.g.\, \cite[Thm.\,3,\,p.\,277]{evans}.
Now, in order to show \eqref{eq:est1},  let $h>0$ be arbitrary, but fixed. In view of \eqref{eq:polar}, it holds
  \begin{equation*}
\ro{   \dual{-\UU'(q_\e(t\pm h))}{\dot q_\e (t)}_Y}+\dual{\ell(t\pm h)}{\dot q_\e(t)}_{Y}-\e( \dot q_\e(t\pm h),\dot q_\e(t))_{Y} \leq \RR(\HH(q_\e)(t\pm h),\dot q_\e(t))   \end{equation*}a.e.\,in $(0,T).$
 Subtracting \eqref{eq:p1} from the above inequality yields
  \begin{equation}\label{hh}\begin{aligned}
 & \ro{  \frac{1}{h} \dual{-\UU'(q_\e(t\pm h))+\UU'(q_\e(t))}{\dot q_\e (t)}_Y}+\dual{\frac{\ell(t\pm h)-\ell(t)}{h}}{\dot q_\e(t)}_Y
  \\&\qquad - \e\Big( \frac{ \dot q_\e(t\pm h)- \dot q_\e(t)}{h},\dot q_\e(t)\Big)_{Y} \leq \frac{1}{h} (\RR(\HH(q_\e)(t\pm h),\dot q_\e(t))-\RR(\HH(q_\e)(t),\dot q_\e(t)) )  \end{aligned} \end{equation}a.e.\,in $(0,T).$
  In light of $q_\e \in H^2(0,T;Y)$, $\ell \in H^1(0,T;Y^\ast)$ and Lemma \ref{bound_k}, we can  pass to the limit $h \searrow 0$ in \eqref{hh}, which gives in turn  
  \begin{equation*}
\ro{  -\UU''(q_\e(t))[\dot q_\e(t)]^2}+\dual{\dot \ell (t)}{\dot q_\e(t)}_Y
-\e ( \ddot q_\e(t), \dot q_\e (t))_{Y} =  {\frac{d}{ds} \Big[\RR(\HH(q_\e)(s),\dot q_\e(t))\Big]\Big|_{s=t}}
\end{equation*}$\ae (0,T).$
Note that Lemma \ref{bound_k} is indeed applicable, since  $q_\e'(t) \in \dom \RR$ f.a.a.\, $t \in (0,T).$ Thus, \ro{by \eqref{uconv2}}, we get
 \[ \alpha \| \dot q_\e(t)\|_{Y}^2 +\frac{\e}{2} \frac{d}{dt}  \| \dot q_\e(t)\|_{Y}^2 \leq \dual{\dot \ell (t)}{\dot q_\e(t)}_{Y}- {\frac{d}{ds} \Big[\RR(\HH(q_\e)(s),\dot q_\e(t))\Big]\Big|_{s=t}}
 \quad  \ae  (0,T).\]
 Integrating over time and using \eqref{kappa} imply
   \begin{equation}\label{eq:syst_dif22}
\begin{aligned}
\alpha \ro{\| \dot q_\e\|_{L^2(0,T;Y)}^2} &+\frac{\e}{2}  ( \| \dot q_\e(T)\|_{Y}^2-\| \dot q_\e(0)\|_{Y}^2)
\\&\quad \leq \|\ell\|_{H^1(0,T;Y^\ast)}   \ro{\| \dot q_\e\|_{L^2(0,T;Y)}}  
  \\&\qquad +C\,\|q_\e\|^2_{C([0,T];Y)}+\frac{\alpha}{4} \|\dot q_\e\|^2_{L^1(0,T;Y)} {+c_0},  \end{aligned}\end{equation}where the constants  {$C,c_0>0$ depend only on $L_\RR, b,\tilde b,B$ and $\alpha$.}
  Further, from Assumption \ref{it:1} we deduce that $\HH(q_\e)(0)=y_0$.
 Since $\ell \in \mathfrak{L}(\ro{q_0,y_0})$ and $q_\e(0)=\ro{q_0}$, one has $-\partial_q \EE(0,q_\e(0)) \in \partial_2 \RR(y_0,0),$  and \eqref{eq:syst_diff1}  leads to $\dot q_\e(0)=0.$ Using this information  in  \eqref{eq:syst_dif22} gives in turn 
   \begin{equation}\label{eq:syst_dif222}
\begin{aligned}
\alpha  \ro{\| \dot q_\e\|_{L^2(0,T;Y)}^2}& \leq  \|\ell\|_{H^1(0,T;Y^\ast)}    \ro{\| \dot q_\e\|_{L^2(0,T;Y)}}  \\&\quad +C\,\|q_\e\|^2_{C([0,T];Y)}+\frac{\alpha}{4} \|\dot q_\e\|^2_{L^1(0,T;Y)} {+c_0}.
  \end{aligned}\end{equation}
The desired assertion now follows by employing 
  Young's inequality and Lemma \ref{prop:c_est}.
  \end{proof}

\section{Existence of solutions for the \ro{rate-independent} system}\label{sec:ex}
With the uniform bounds established in the last section at hand, we can now 
 focus on showing that the \ro{rate-independent} evolution
\begin{equation}\label{eq:n1}
  \begin{gathered}
-\partial_q \EE(t,q(t)) \in \partial_2 \RR (\HH(q)(t),\dot{q}(t)) \quad \ae (0,T), \quad q(0) =\ro{q_0}
  \end{gathered}
  \end{equation} admits solutions, \ro{where $q_0 \in Y$ is a fixed initial value}.  {Even though our energy is convex and the usual technique would involve a time-discretization \cite{mielke}, we choose to prove the existence result for \eqref{eq:n1} by letting $\e \to 0$ in the viscous formulation provided by Lemma \ref{prop:ei}. Recall that we proceed in this manner  due to future investigations regarding the optimal control of \eqref{eq:n1}, see Remark \ref{rem:fw} below. }

First, we need to formulate an assumption concerning the convergence behaviour of the  dissipation potential \ro{$\RR$ and of the history operator $\HH$}. We point out that this will also be essential for the existence of optimal solutions to \eqref{eq:min} below, see proof of Theorem \ref{ex_opt}.

\begin{assumption}\label{assu:r}
Suppose that $\dom \RR\subset Y$ is a closed set. Let $y_m \weakly y$ in $H^1(0,T;Y)$, where $\{y_m\}  \subset  \YY:=\{y \in H^1(0,T;Y):y(0)=q_0,\ \dot y (t) \in \dom \RR \text{ f.a.a.\,}t \in (0,T)\}$. Then,
\begin{enumerate}
\item\label{m1}
It holds 
 \begin{equation*}\begin{aligned}
\liminf_{m \to \infty}  \int_0^T \RR(\HH(y_m)(t),\dot y_m(t))  \,dt  \geq  \int_0^T \RR(\HH(y)(t),\dot y(t))  \,dt ;\end{aligned}\end{equation*}
\item \label{m2} \ro{For each $v \in  L^2(0,T;\dom \RR )$ there exists a sequence $\{v_m\} \subset  L^2(0,T;\dom \RR )$ with  $v_m \weakly v$ in $L^2(0,T;Y)$ and 
\begin{align*}\limsup_{m \to \infty} & \int_0^T \RR(\HH(y_m)(t), v_m(t)) +\UU'(y_m(t))(v_m(t)) \,dt  \\
&\leq  \int_0^T \RR(\HH(y)(t), v(t)) +\UU'(y(t))(v(t)) \,dt.\end{align*}}
\end{enumerate}
\end{assumption}
\begin{remark}\label{rem:rr}
We observe that, since $\dom \RR$ is assumed to be a closed subset of $Y$, the set $\YY$ is weakly closed in $H_0^1(0,T;Y)$, so that $y \in \YY$ in Assumption \ref{assu:r}; note that  $\dom \RR$ is convex, in view of the convexity of $\RR$ with respect to its second argument (Assumption \ref{it:st1}.\ref{it:st11}). \end{remark}
\begin{remark}
We underline that Assumption \ref{assu:r} is satisfied by all the  functionals  from Remark \ref{rem:r},  {provided that $Y \embed \embed X$, $B(t,\cdot):X \to X$ is continuous for each $t \in [0,T]$ and if a requirement such as Assumption \ref{assu:rd} holds true}. See the proof of Theorem \ref{ex1} below, where  the case of the dissipation functional \eqref{def:r} from section \ref{u0} is addressed.
\\To see that Assumption \ref{assu:r} is also fulfilled by the functional $\widehat R:L^2(\G_2) \times H^1_{\G_1}(\O) \to \R$, 
\[\widehat R(\zeta,\eta)=\int_{\G_2} g(\zeta) | \eta|\,ds\]
appearing in the modelling of \ro{history-dependent} evolutions with Tresca's friction (Remark \ref{rem:r}),
one uses the assumed continuity of $B(t,\cdot):L^2(\G_2)\to L^2(\G_2)$ and that of $g$,  see Remark \ref{rem:r}, to show
\[g(\HH(y_m)) \to g(\HH(y)) \quad \text{in }L^2(0,T;L^2(\G_2)) \quad \text{as }m \to \infty
.\] Then, \ro{if $\UU':L^2(0,T;H^1_{\G_1}(\O))\to L^2(0,T; H^1_{\G_1}(\O)^\ast)$ is weakly  continuous},  Assumption \ref{assu:r}.\ref{m2} holds true by setting $v_m:=v$, while Assumption \ref{assu:r}.\ref{m1} is due to \[\dot y_m \weakly \dot y \quad \text{in } L^2(0,T;L^2(\G_2))  \quad \text{as }m \to \infty\] and \cite[Thm.\,3.23]{dac08}. The same applies to the functional defined in \eqref{def:t0}.
\end{remark}

%

 {We are now in the position to formulate one of the main results of this paper.}

\begin{theorem}[Solvability of the \ro{history-dependent},\,\ro{rate-independent}  problem]\label{ex}
Let \ro{Assumptions \ref{assu:ex} and  \ref{assu:r} hold}. Then, for each
 $\ell \in H^{1}(0,T;Y^\ast) \cap \mathfrak{L}(\ro{q_0,y_0})$,  the \ro{rate-independent} evolution \eqref{eq:n1} admits at least one  solution $q \in \ro{\hy}$. 
\end{theorem}

\begin{proof}
Let $\e>0$ be small enough, as in Assumption \ref{assu:ex}, and let $q_\e \in \ro{H^1(0,T;Y)}$ be a solution to \eqref{eq:q1} with right hand side\ $\ell.$ From Proposition \ref{lem:est} we know that there exists a (not relabelled) subsequence of $\{q_\e\}$ and some $q \in \ro{H^1(0,T;Y)}$ so that 
 \begin{equation}\label{eq:h1}
 q_\e \weakly q \quad \text{in }\ro{H^1(0,T;Y)} \quad \text{as }\e \searrow 0,
  \end{equation}\ro{which means that
   \begin{equation}\label{eq:h1p}
 q_\e(t) \weakly q(t) \quad \text{in }Y \quad \text{as }\e \searrow 0, \quad \forall\,t \in [0,T].
  \end{equation}
According to  Lemma \ref{prop:ei}, $q_\e$ satisfies
 \begin{equation}\label{eq:eii}\begin{aligned}
\int_0^t  \RR (\HH(q_\e)(\tau),\dot q_\e(\tau)) d \tau +\int_0^t &\e \| \dot q_\e (\tau)\|^2_{Y} d \tau + \EE(t,q_\e(t))
\\& =\EE(0,q_0)+  \int_0^t  \partial_t \EE(\tau,q_\e(\tau))  d \tau \quad \forall\ t \in [0,T].\end{aligned}
\end{equation}
 As a  consequence of \eqref{def:e}, combined with the weakly lower semicontinuity of $\UU$ (Assumption \ref{it:0}), the linearity of $\ell(t)$ and \eqref{eq:h1p}, we have 
  \begin{equation}\label{eq:h1_t}
 \EE(t,q(t)) \leq \liminf_{\e \searrow 0} \EE(t,q_\e(t)) \quad \forall\, t \in [0,T].\end{equation}
Thus, thanks to\ Assumption \ref{assu:r}, we can deduce from \eqref{eq:eii} with $t:=T$ that
 \begin{equation}\label{for_l1}\begin{aligned}
 \int_0^T \RR(\HH(q)(\tau),\dot q(\tau))\,d\tau& + \EE(T,q(T))
 \\&\leq  \liminf_{\e \searrow 0} \int_0^T \RR(\HH(q_\e)(\tau),\dot q_\e(\tau))\,d\tau  +\liminf_{\e \searrow 0} \EE(T,q_\e(T))
\\&\leq \EE(0,q_0)+\lim_{\e \searrow 0}  \int_0^T  \partial_t \EE(\tau,q_\e(\tau))  d \tau  
\\&\quad=\EE(0,q_0)+ \int_0^T \partial_t \EE(\tau,q(\tau))  d \tau,
\end{aligned}\end{equation}where the last identity is due to \eqref{deriv} and \eqref{eq:h1}. Applying \eqref{chain} integrated over $[0,T]$ then yields
 \begin{equation}\label{for_l11}\begin{aligned}
 \int_0^T \RR(\HH(q)(\tau),\dot q(\tau))\,d\tau \leq \int_0^T -\dual{\partial_q \EE(\tau,q(\tau))}{\dot q(\tau)}_Y  d \tau.
\end{aligned}\end{equation}}
\ro{Our next goal is to show \eqref{rstar00} below.
To this end, let $v \in L^2(0,T;\dom \RR )$ be arbitrary, but fixed and  denote by $\{v_\e\}$ the recovery sequence from Assumption \ref{assu:r}.\ref{m2}.
By employing  \eqref{eq:polar}  and \eqref{deriv} we then deduce 
 \begin{equation}\label{for_l}\begin{aligned}
\limsup_{\e \searrow  0} & \int_0^T -\e(\dot q_\e(t), v_\e(t))_Y+\dual{\ell(t)}{v_\e(t)}_Y \,dt 
\\& \leq \limsup_{\e \searrow  0} \int_0^T \RR(\HH(q_\e)(t), v_\e(t))+\UU'(q_\e(t))(v_\e(t)) \,dt  \\
&\leq  \int_0^T \RR(\HH(q)(t), v(t)) +\UU'(q(t))(v(t)) \,dt.
 \end{aligned}\end{equation} 
 In view of $v_\e \weakly v$ in $L^2(0,T;Y)$, by Assumption \ref{assu:r}.\ref{m2}, and \eqref{eq:h1}, we infer that the term $\dual{\dot q_\e}{ v_\e}_{L^2(0,T;Y)}$ is uniformly bounded w.r.t.\,$\e$. From \eqref{for_l} one then has 
 \begin{equation}\begin{aligned}
 \int_0^T \dual{\ell(t)}{v(t)}_Y \,dt  \leq   \int_0^T \RR(\HH(q)(t), v(t)) +\UU'(q(t))(v(t)) \,dt
 \end{aligned}\end{equation}  for all $v\in L^2(0,T;Y)$. Let now $\varphi \in C^\infty[0,T], \varphi \geq 0,$ and $\eta \in \dom \RR$ be arbitrary but fixed. 
Testing with $v:=\varphi \eta$ and using the positive homogeneity of $\RR(\HH(q(t)),\cdot)$ (Assumption \ref{it:st1}.\ref{it:st11}), as well as the fundamental lemma of calculus of variations, and \eqref{deriv}, give in turn 
\begin{equation}\label{rstar}
\dual{  -\partial_q \EE(t,q(t)) }{\eta}_{Y}\leq \RR(\HH(q)(t),\eta) \quad \forall\, \eta \in \dom \RR,  \quad   \ae (0,T).
\end{equation}
Thanks to \eqref{for_l11}, \eqref{rstar} tested with $\dot q(t)$ leads to 
\begin{equation}\label{rstar00}
\dual{  -\partial_q \EE(t,q(t)) }{\dot q(t)}_{Y} = \RR(\HH(q)(t),\dot q(t)) \quad    \ae (0,T).
\end{equation}
By subtracting \eqref{rstar00} from \eqref{rstar}, we can conclude that the weak limit $q$ from \eqref{eq:h1} is a solution to \eqref{eq:n1}; note that $q(0)=q_0,$ by \eqref{eq:h1p}. This competes the proof.}
 \end{proof}

We end this section with some useful equivalent formulations of \eqref{eq:n1}. As a  result
of the positive homogeneity of $\RR$ from Assumption \ref{it:st1}.\ref{it:st11}, we have the following
\begin{lemma}\label{lem:cor}
Let $\ell \in H^{1}(0,T;Y^\ast) \cap \mathfrak{L}\ro{(q_0,y_0)}$. Then, each solution \ro{$q \in H^1(0,T;Y)$  } \ro{associated with }\eqref{eq:n1} satisfies 
 \begin{subequations}  \begin{gather}
\dual{  -\partial_q \EE(t,q(t)) }{\dot q(t)}_{Y}=\RR(\HH(q)(t),\dot q(t)),\label{eq:p11}
\\  \dual{  -\partial_q \EE(t,q(t)) }{v}_{Y}\leq \RR(\HH(q)(t),v) \quad \forall\, v \in Y,  \quad   \ae (0,T). \label{eq:polar1}\end{gather}\end{subequations}
\end{lemma}
We remark that  \ro{Lemma  \ref{lem:cor}} corresponds to Lemma \ref{prop:equiv}.  \ro{We also arrive at a counterpart of the energy identity from Lemma \ref{prop:ei}, which is stated next. This will be used in the upcoming section to prove Theorem \ref{ex_opt}, as well as the boundedness of the control-to-state operator (Lemma \ref{lem:est_ris}).
\begin{proposition}[Energy identity]
\label{lem:ei}
Let $\ell \in H^{1}(0,T;Y^\ast) \cap \mathfrak{L}\ro{(q_0,y_0)}$. Then, each solution \ro{$q \in H^1(0,T;Y)$  } associated with \eqref{eq:n1} fulfills 
\begin{equation}\label{eq:ei_ris}\begin{aligned}
\int_0^t  \RR(\HH(q)(\tau),\dot q(\tau)) d \tau + \EE(t,q(t))
 =\EE(0,q(0))+  \int_0^t  \partial_t \EE(\tau,q(\tau))  d \tau \ \ \forall\,0 \leq t\leq T.\end{aligned}
\end{equation}
 \end{proposition}
\begin{proof}The result follows by integrating  \eqref{eq:p11} over $[0,t]$ and applying  \eqref{chain}.\end{proof}
}

 {We end our discussion concerning the existence of solutions for \eqref{eq:n1} by pointing to section \ref{u0}. There, we will obtain uniqueness results for a particular class of problems \eqref{eq:n1} which feature infinite-valued dissipation functionals that display linear behaviour in their second argument. }

\section{Existence of optimal solutions for the control problem}\label{oc}
The result in Theorem \ref{ex} enables us to investigate the existence of optimal solutions for control problems governed by \eqref{eq:n1}, as we will next see. 
To be able to make use of Theorem \ref{ex},  {Assumptions \ref{assu:ex} and  \ref{assu:r} are tacitly assumed in this section, without mentioning them  every time.} We are concerned with  the following   minimization problem
 \begin{equation}\label{eq:min}
 \left.
 \begin{aligned}
  \min \quad & j(q)+\frac{1}{2}\|\ell\|^2_{H^1(0,T;Z)}\\
  \text{s.t.} \quad & (\ell,q) \in \ro{{H^1(0,T;Z)}\times {H^1(0,T;Y)}},
  \\\quad & q \text{ solves }
\eqref{eq:n1} \text{ with right hand side\ }\ell \in  \mathfrak{L}\ro{(q_0,y_0)} , \end{aligned}
 \quad \right\}
\end{equation}where $Z$ is a real reflexive Banach space so that $Z \embed  \embed Y^\ast$ and {$j:H^1(0,T;Y) \to \R$ is a weakly lower semicontinuous mapping, bounded from below.}  {We recall that $q_0\in Y$ is the initial datum in \eqref{eq:n1}, while $y_0\in X$ is given by Assumption \ref{it:1}. The compatibility set $\mathfrak{L}(q_0,y_0)$ was introduced  in Definition \ref{def} as \[\mathfrak{L}({q_0,y_0}):=\{\ell \in C([0,T];Y^\ast):{-\UU'(q_0)+}\ell(0) \in \partial_2 \RR(y_0,0) \}.\]
Throughout this section, we denote the admissible set of \eqref{eq:min} by  
\[{\MM:=\{(\ell,q) \in {H^1(0,T;Z)} \times {H^1(0,T;Y)}: q \text{\ solves\ }\eqref{eq:n1} \text{ { with r.h.s.\,}}\ell \in  {\mathfrak{L}(q_0,y_0)}\}.}\]
Note that $\MM$ is not empty, in view of Theorem \ref{ex}.}

 {Even though the unique solvability of  \eqref{eq:n1} is not available, we can  apply the classical direct method of calculus of variations to prove existence of optimal solutions for \eqref{eq:min}, since  $\MM$  is weakly sequentially closed (in  ${H^1(0,T;Z)}\times {H^1(0,T;Y)}$). That is, for each sequence $\{(\ell_n,q_n)\}\subset \MM$, it holds  \begin{equation}\label{wc}(\ell_n,q_n)\weakly (\ell,q) \quad \text{in }{H^1(0,T;Z)}\times {H^1(0,T;Y)} \text{ as }n \to \infty \Longrightarrow (\ell,q)\in \MM.\end{equation}The convergence  \eqref{wc} is due to the energy identity from Proposition \ref{lem:ei} and \eqref{eq:polar1}, cf.\,the proof of Theorem \ref{ex_opt} below.}

 {While applying the direct method of calculus of variations, we need to be able to extract weakly convergent (sub)sequences as in \eqref{wc}. In this context, we will rely on the coercive nature (w.r.t.\,$\ell$) of the objective in \eqref{eq:min} combined with the boundedness of  the (multivalued) solution operator of \eqref{eq:n1}. The latter is established  next.} 

\begin{lemma}\label{lem:est_ris}
For every $\ell \in H^1(0,T;Y^\ast) \cap \mathfrak{L}\ro{(q_0,y_0)}$, the  solutions $q \in   H^1(0,T;Y) $ to \eqref{eq:n1}  satisfy \begin{equation}\label{h1_est}
\|q\|_{\ro{\hy}} \leq c\|\ell\|_{H^1(0,T;Y^\ast)}\ro{+c_0},
 \end{equation}where $\ro{c,c_0>0}$ are independent of $\ell.$
   \end{lemma} 

   \begin{proof}
   By arguing exactly as in  the proof of Lemma \ref{prop:c_est}, where this time one uses  \eqref{eq:ei_ris}, we obtain the estimate
\begin{equation}\label{eq:c_ris}
\|q\|_{C([0,T];Y)} \leq C\,\|\ell\|_{W^{1,1}(0,T;Y^\ast)}\ro{+c_0},
 \end{equation}where $C,c_0>0$ are independent of $\ell.$
From \eqref{eq:polar1} and arguments employed at the beginning of the proof of Proposition \ref{lem:est} we further deduce
 \begin{equation*}
\ro{-\UU''(q(t))[\dot q(t)]^2}+\dual{\dot \ell (t)}{\dot q(t)}_Y
 = {\frac{d}{ds} \Big[\RR(\HH(q)(s),\dot q(t))\Big]\Big|_{s=t}}
\end{equation*}$\ae (0,T).$
Again, we note that $\dot q (t) \in \dom \RR$ f.a.a.\,$t \in (0,T)$, which allows us to apply Lemma \ref{bound_k} for $y:=q$. Using \eqref{uconv2}, integrating the above identity over time and  \eqref{kappa}, imply
   \begin{equation*}
\begin{aligned}
{\alpha}\ro{\|\dot q\|_{L^2(0,T;Y)}^2} &\leq \|\ell\|_{H^1(0,T;Y^\ast)} \ro{\|\dot q\|_{L^2(0,T;Y)}} 
  \\&\quad -\int_0^T  {\frac{d}{ds} \Big[\RR(\HH(q)(s),\dot q(t))\Big]\Big|_{s=t}}\,dt
\\&  \leq \|\ell\|_{H^1(0,T;Y^\ast)} \ro{\|\dot q\|_{L^2(0,T;Y)}}
\\&\quad +C\,\|q\|^2_{C([0,T];Y)}+\frac{\alpha}{4} \|\dot q\|^2_{L^1(0,T;Y)} {+c_0},
  \end{aligned}\end{equation*}which results in 
   \begin{equation*}
\begin{aligned}
\ro{\|\dot q\|_{L^2(0,T;Y)}^2} \leq c\,\|\ell\|^2_{H^1(0,T;Y^\ast)}+\tilde C\,\|q\|^2_{C([0,T];Y)} {+c_0},  \end{aligned}\end{equation*}
where $c,\tilde C, {c_0}>0$ are independent of $\ell$.
Employing \eqref{eq:c_ris}  in the above estimate then yields the desired assertion.
  \end{proof}
 {We now have the necessary ingredients to prove the main result of this section.}
\begin{theorem}[Existence of optimal solutions]\label{ex_opt}
Under \ro{Assumptions \ref{assu:ex} and \ref{assu:r}}, the optimization problem \eqref{eq:min} admits at least one solution.
\end{theorem}
\begin{proof}
 {The proof makes use of the direct method of calculus of variations, Lemma \ref{lem:est_ris} and \eqref{wc}. Similar arguments may be found for instance  in \cite{fr,fr_a,kms}, where the authors also deal with the existence of optimal solutions for problems governed by state systems with multivalued solution operators.}
We recall that the admissible set is given by 
\[\ro{\MM=\{(\ell,q) \in {H^1(0,T;Z)} \times {H^1(0,T;Y)}: q \text{\ solves\ }\eqref{eq:n1} \text{ \ro{ with r.h.s.\,}}\ell \in  \ro{\mathfrak{L}(q_0,y_0)}\}}\]
and that is not empty, in view of Theorem \ref{ex}.
Let $\{(\ell_n, q_n)\} \subset \MM$ be an infimal sequence, i.e.,
\begin{equation}\label{eq:inf}
j(q_n)+\frac{1}{2}\|\ell_n\|^2_{H^1(0,T;Z)} \overset{n \to \infty}{\to} \inf \{j(q)+\frac{1}{2}\|\ell\|^2_{H^1(0,T;Z)}: (\ell,q) \in \MM\}.
\end{equation}
{Since $j$ is bounded from below, by assumption, we can select a subsequence so that }
\begin{equation}\label{eq:ell}
\ell_n \weakly \ell \quad \text{in }\ro{H^1(0,T;Z)} \quad \text{as }n \to \infty,
\end{equation}
\ro{which in particular leads to
\begin{equation}\label{eq:ce}
\ell_n \to  \ell \quad \text{in }L^2(0,T;Y^\ast) \quad \text{as }n \to \infty,
\end{equation}due to the compact embedding $H^1(0,T;Z) \embed \embed L^2(0,T;Y^\ast)$. Recall that $Z \embed \embed Y^\ast,$ by assumption.
We observe that $\ell  \in \mathfrak{L}(q_0,y_0),$ owing to  \eqref{eq:ell} and $\ell_n \in \mathfrak{L}(q_0,y_0)$.
In view of Lemma \ref{lem:est_ris} and \eqref{eq:ell} combined with $Z \embed Y^\ast$,  there exists $q \in \ro{H^1(0,T;Y)}$  so that
\begin{equation}\label{eq:q_n}
q_n \weakly  q \quad \text{in }\ro{H^1(0,T;Y)} \quad \text{as }n \to \infty, \quad \ro{q(0)=q_0}.\end{equation}
Now, our aim is to show that $q \in \hy$ satisfies \eqref{eq:n1} with right hand side\ $\ell$.
\ro{Since $\ell  \in \mathfrak{L}(q_0,y_0)$, this will ensure the feasibility of $(\ell,q)$.} 
To this end, we argue as in the proof of 
Theorem \ref{ex}, where this time we rely on \eqref{eq:ei_ris} and \eqref{eq:polar1}.
As $q_n$ solves \eqref{eq:n1} (with right hand side\, $\ell_n$), we have the energy identity
\begin{equation}\label{eq:ei_n}\begin{aligned}
\int_0^T  \RR(\HH(q_n)(\tau),\dot q_n(\tau)) d \tau &+ \UU(q_n(T))-\dual{\ell_n(T)}{q_n(T)}_Y
\\& = \UU(q_0)-\dual{\ell_n(0)}{q_0}_Y -  \int_0^T \dual{\dot \ell_n(\tau)}{q_n(\tau)}_Y
  d \tau,\end{aligned}
\end{equation}
in view of Lemma \ref{lem:ei}, \eqref{def:e} and \eqref{deriv}. 
Then, we use the arguments for the proof of \eqref{eq:h1_t}, \eqref{eq:ell} and the compact embedding $Z \embed \embed Y^\ast$, as well as Assumption \ref{assu:r}.\ref{m1}, to obtain
 \begin{equation}\begin{aligned}
 \int_0^T &\RR(\HH(q)(\tau),\dot q(\tau))\,d\tau + \UU(q(T))-\dual{\ell(T)}{q(T)}_Y
 \\&\leq  \liminf_{n \to \infty} \int_0^T \RR(\HH(q_n)(\tau),\dot q_n(\tau))\,d\tau  +\liminf_{n \to \infty} \UU(q_n(T))-\lim_{n\to \infty} \dual{\ell_n(T)}{q_n(T)}_Y
\\&\leq \UU(q_0)-\lim_{n\to \infty} \dual{\ell_n(0)}{q_0}_Y - \lim_{n\to \infty} \int_0^T \dual{\dot \ell_n(\tau)}{q_n(\tau)}_Y \,d\tau
\\&\quad=\UU(q_0)-\dual{\ell(0)}{q_0}_Y +  \int_0^T \dual{\dot \ell(\tau)}{q(\tau)}_Y\,d \tau .\end{aligned}\end{equation}Note that the last identity is due to \eqref{eq:q_n}, the compact embedding $H^1(0,T;Y) \embed \embed L^2(0,T;Z^\ast)$ and \eqref{eq:ell}. Applying \eqref{chain} integrated over $[0,T]$ then yields
 \begin{equation}\label{qq}\begin{aligned}
 \int_0^T \RR(\HH(q)(\tau),\dot q(\tau))\,d\tau \leq \int_0^T -\dual{\partial_q \EE(\tau,q(\tau))}{\dot q(\tau)}_Y  d \tau.
\end{aligned}\end{equation}
To show \eqref{rstar3} below, let $v \in L^2(0,T;\dom \RR )$ be arbitrary, but fixed and  denote by $\{v_n\}$ the recovery sequence from Assumption \ref{assu:r}.\ref{m2}.
By employing  \eqref{eq:polar1}  and \eqref{deriv} associated with $\ell_n$, we deduce 
 \begin{equation}\label{for1}\begin{aligned}
\limsup_{n \to \infty} & \int_0^T \dual{\ell_n(t)}{v_n(t)}_Y \,dt 
\\& \leq \limsup_{n \to \infty} \int_0^T \RR(\HH(q_n)(t), v_n(t))+\UU'(q_n(t))(v_n(t)) \,dt  \\
&\leq  \int_0^T \RR(\HH(q)(t), v(t)) +\UU'(q(t))(v(t)) \,dt.
 \end{aligned}\end{equation} 
 In view of $v_n \weakly v$ in $L^2(0,T;Y)$, by Assumption \ref{assu:r}.\ref{m2}, and \eqref{eq:ce}, we infer from \eqref{for1} that 
 \begin{equation}\begin{aligned}
 \int_0^T \dual{\ell(t)}{v(t)}_Y \,dt  \leq   \int_0^T \RR(\HH(q)(t), v(t)) +\UU'(q(t))(v(t)) \,dt
 \end{aligned}\end{equation}  for all $v\in L^2(0,T;Y)$. 
By arguing as in the proof of \eqref{rstar}, we  obtain 
\begin{equation}\label{rstar3}
\dual{  -\partial_q \EE(t,q(t)) }{\eta}_{Y}\leq \RR(\HH(q)(t),\eta) \quad \forall\, \eta \in \dom \RR,  \quad   \ae (0,T).
\end{equation}
The estimate \eqref{rstar3} tested with $\dot q(t)$ together with \eqref{qq} allows us to conclude that $q \in \hy$ satisfies \eqref{eq:n1} with right hand side\ $\ell$, i.e., $(\ell,q) \in \MM.$} Since \[(\ell_n,q_n)  \weakly (\ell,q) \quad \text{in }H^1(0,T;Z) \times H^1(0,T;Y),\] by \eqref{eq:ell} and \eqref{eq:q_n}, and since  $j:H^1(0,T;Y) \to \R$ is weakly lower semicontinuous, by assumption, one arrives at 
 \[ j(q)+\frac{1}{2}\|\ell\|_{H^1(0,T;Z)}^2 \leq \liminf_{n \to \infty} j(q_n)+\frac{1}{2}\|\ell_n\|_{H^1(0,T;Z)}^2.
 \]
 As $(\ell,q) $ is admissible for \eqref{eq:min}, we deduce in light of \eqref{eq:inf} that $(\ell,q) $ is an optimal solution to \eqref{eq:min}. The proof is now complete.
\end{proof}

 {We end this section with some comments concerning comparable contributions.
\begin{remark}The proof of Theorem \ref{ex_opt} shows that $\MM$ is weakly sequentially closed in $H^1(0,T;Z)\times H^1(0,T;Y)$, i.e, it satisfies \eqref{wc}. In \cite{kms}, the authors also rely on this property of the admissible set to show existence of optimal controls for problems governed by RIS with non-convex energies. Therein, the notion of ``vanishing-viscosity" solution is used. By contrast, \cite{fr} deals with optimization problems where the set of solutions for the state system consists of energetic solutions. We point out that this work contains the   first result concerning the existence of optimal solutions for minimization problems governed by RIS with non-convex energies. We remark that, as our energy is uniformly convex, the concepts ``vanishing-viscosity" and ``energetic"  solution coincide with the classical differential one \cite{mielke}, that is, the notion we use in our Theorem \ref{ex}. 
A more general setting than in \cite{fr} can be found in \cite{fr_a} where optimal solutions are constructed via an approximation scheme in an abstract framework  \cite[Sec.\,2]{fr_a} that applies to RIS with non-convex energies \cite[Sec.\,3-4]{fr_a}.
\\If the convergence in \eqref{eq:ell} were strong, \eqref{wc} could be interpreted as an ``upper semicontinuity" property of our multivalued solution operator \cite{fr,fr_a}.  This is precisely the case  in \cite{fr}, where the set  for optimal controls is assumed to be compact  \cite[(J1)]{fr}, such that the existence of strong convergent sequences is ensured \cite[Proof of Thm.\,3.4]{fr}. \end{remark}
\begin{remark}[Infinite-dimensional MPEEC]
We mention that \eqref{eq:min} can be seen as an infinite-dimensional MPEEC {``mathematical programm with evolutionary equilibrum constraints"}. Finite dimensional MPEECs with applications to delamination and micromagnetics have already been investigated in  \cite{mpeec,del}, see also the references therein. Effective numerical methods  have been proposed in \cite{mpeec} in the context where (in each time step) the control-to-state map is single-valued, which allows for the use of the standard implicit programming approach. At the time-discrete level, the solution of the previous state system enters the present state system \cite[Eq.\,(3.10)]{mpeec} as a parameter, so that there is reason to believe that the results in \cite{mpeec} can be extended to applications  where the respective continuous model takes a memory variable into account. In this situation, the only difference would be that all the previous unique solutions of (time-discrete) state systems would enter the present problem as parameters.  On the other hand, if we look at the purely continuous version of the model in \cite{mpeec}, which becomes an infinite-dimensional MPEEC, then  the methods from the proof of Theorem \ref{ex_opt} may be applicable as long as the admissible set is weakly sequentially closed.
\end{remark}}

%
\section{Unique solvability for a class of \ro{rate-independent} systems}\label{u0}

In this section, the focus lies on proving uniqueness of solutions for \ro{history-dependent} evolutions of the type \eqref{eq:n}, 
where the dissipation potential is linear  with respect to its second argument. Such dissipation potentials  include functionals \ro{with unbounded subdifferential}, as we will see below.
For convenience of the reader, we will work in the rest of the paper with concrete Sobolev spaces, although a general framework, as in the previous sections, may also be considered.

We begin by fixing the  setting. Let $\O \subset \R^N$, $N \in \{2,3\}$, be a bounded Lipschitz domain.
The problem we investigate reads as follows
\begin{equation}\label{eq:n2}
  \begin{gathered}
-\partial_q \EE(t,q(t)) \in \partial_2 R (\HH(q)(t),\dot{q}(t)) \quad \ae (0,T), \quad q(0) = 0,
  \end{gathered}
  \end{equation} \ro{that is, $q_0=0$ in the sequel, for convenience.}
The dissipation functional $R:\lo \times \hon \to [0,\infty]$ is given by 
 \begin{equation}\label{def:r}
R(\zeta,\eta):=\left\{ \begin{aligned}\int_\Omega \kappa(\zeta) \,  \eta \;dx  &\quad \text{if }\eta \in \CC,
\\\infty  &\quad \text{otherwise,}\end{aligned} \right.
\end{equation}where \begin{equation}\label{def:c}
\CC:=\{\eta \in \hon:\eta \geq 0 \text{ a.e. in }\Omega \}.
\end{equation}
In the rest of the paper, the following assumption is supposed to hold true, without mentioning it every time.
\begin{assumption}\label{assu:c}
For the mappings involved   in \eqref{eq:n2} we require:
 \begin{enumerate}

   \item\label{it:c0} Given a  time-dependent load $\ell:[0,T] \to \hoon$, 
the energy $\EE:[0,T] \times \hon \rightarrow \mathbb{R}$ is defined
as 
 \ro{
\begin{equation}\label{def:e1}\begin{aligned}
\EE(t, y)&:=\UU(y)-\dual{\ell(t)}{y}_{\hon},
\end{aligned}\end{equation}where $\UU:\hon \to \R$ is  uniformly convex, i.e., there exists $\alpha>0$ so that $\forall\, y_1,y_2 \in  \hon, \forall\,\theta \in [0,1]$ it holds
 \begin{equation}\label{uconv0}(1-\theta)\UU(y_1)+\theta \UU(y_2)\geq \UU(y_\theta)+ \frac{\alpha}{2} \theta (1-\theta)\|y_2-y_1\|_{\hon}^2,\end{equation}where $y_\theta:=\theta y_1+(1-\theta)y_2.$ Moreover, $\UU \in C^2(\hon)$.
}
\item \label{it:c1}
The non-linear function  $\kappa: \R \to [0,\infty)$ is assumed to be Lipschitz continuous with Lipschitz constant $L_\kappa>0$. 
 \item \label{it:c2}  The \textit{history operator} $\HH:L^1(0,T;H^1(\O)) \to W^{1,1}(0,T;\lo)$ is given by
 \begin{equation*}
 [0,T] \ni t \mapsto \HH(y)(t):=\int_0^t  {B(t-s,y(s))} \,ds +y_0 \in \lo,\end{equation*}
where $y_0\in \lo$ is fixed and  {$B:[0,T]\times \hon\to \lo$ satisfies 
\begin{itemize}
\item for almost all $t\in (0,T),$ the mapping $B(t,\cdot):\lo \to \lo$ is continuous and $B(t,\cdot):\hon \to \lo$ is Lipschitzian with Lipschitz-constant $b(t)$ satisfying $b \in L^2(0,T);$ in addition, $B(\cdot,0)\in L^1(0,T;\lo);$
\item for each $y \in \hon$ the mapping $B(\cdot,y):[0,T] \to \lo$ is differentiable and $\partial_t B(\cdot,y):[0,T] \to \lo$ is measurable;
\item for almost all $t\in (0,T),$ the mapping $\partial_t B(t,\cdot):\hon \to \lo$ is Lipschitz-continuous with Lipschitz-constant $\tilde b(t)$ satisfying $\tilde b \in L^1(0,T);$ in addition, $\partial_t B(\cdot,0)\in L^1(0,T;\lo)$.
\end{itemize}}
 \end{enumerate}
\end{assumption} 

\begin{remark}
The structure of the evolution \eqref{eq:n2} is inspired by damage models with fatigue (cf.\, \cite{alessi, alessi1} and the references therein).   In this context, the time-dependent load $\ell:[0,T] \to \hoon$ appearing in \eqref{def:e1}  acts on the body $\O$ and  induces a certain damage, which  is expressed in terms of the variable $q : [0,T] \to H^{1}(\Omega)$. The history operator $\HH$ models how the  damage experienced by the material affects its fatigue level, while the degradation mapping $\kappa:\R \to [0,\infty)$ appearing in \eqref{def:r} indicates in which measure the fatigue affects the toughness  of the material. The latter is usually described by a fixed  (nonnegative) constant \cite{fn96, FK06}, while in the present model it changes  in time, depending on $\HH(q)$. To be more precise, the value of the  toughness of the body at time point $t \in [0,T]$ is given by  $\kappa(\HH(q))(t)$, cf.\ \eqref{eq:n2} and \eqref{def:r}. Hence, damage  models with a history variable take into account the following crucial aspect: the occurrence of damage is favoured in regions where fatigue accumulates. 
\end{remark}

\begin{remark}
[The model \eqref{eq:n2} as \ro{state-independent} RIS with non-convex energy]\label{non-conv}
We remark that 
 \eqref{eq:n2} can be rewritten in terms of the following equivalent 
formulation \begin{equation}\label{eq:ic_ris0}  
-\partial_q \EE(t,q(t))-\F(q)(t) \in \partial \II_\CC(\dot q(t)) \quad  \ae (0,T).\end{equation}This is just the dual formulation of \eqref{eq:hds0}, see Remark \ref{rem:r}.  
The
 differential inclusion \eqref{eq:ic_ris0} has the advantage that the ``dissipation functional" becomes the indicator functional of a \textit{fixed set}, namely $\CC$, cf.\,\eqref{def:c}. We emphasize that, even when $\EE$ is quadratic, 
one cannot argue  that the term appearing on the left hand side, i.e., $-\partial \widehat \EE(q)(t):=-\partial_q \EE(t,q(t))-\F(q)(t)$, satisfies
\begin{equation}\label{non_c} 
\begin{aligned}
&\dual{\partial \widehat \EE(q_1)(t)-\partial \widehat \EE(q_2)(t)}{q_1(t)-q_2(t) }_{\hon}
\\&=\alpha \|(q_1-q_2)(t)\|_{\hon}^2+\dual{\F(q_1)(t)-\F(q_2)(t)}{q_1(t)-q_2(t)}
\\&\qquad \geq \beta \|(q_1-q_2)(t)\|^2 \quad  \ae (0,T),
\quad \forall\,q_1, q_2 \in \hhyy,
\end{aligned}\end{equation}
where $\beta >0$ and $\|\cdot\|$ is a norm in a suitable space. The estimate in \eqref{non_c} is  in general not true. This is due to the fact that $q \mapsto \F(q)$ is monotonically decreasing in applications, since the level of  toughness of the material $\F(q)$ decreases as the value of the damage $q$ increases. Thus, a relation reminiscent of the uniform convexity of the underlying ``energy" \cite{mt2004}, such as \eqref{non_c}, is not to be expected here. In fact, if the Lipschitz constant of $\kappa$ is sufficiently large and $\alpha$ is small enough, \eqref{non_c} is violated. Thus, when rephrased as \eqref{eq:ic_ris0}, \eqref{eq:n2} can be seen as a RIS with non-convex energy and \ro{state-independent} ``dissipation functional" $\II_\CC$.  
\\We finally  point out that, if $\HH=\mathbb{I}$ and the smallness assumption $L_\kappa <\alpha$ (that ensures existence of solutions, see Remark \ref{rem:bound}) is satisfied, \ro{ then \eqref{non_c} is replaced by \begin{equation} 
\begin{aligned}
&\dual{\partial \widehat \EE(q_1)(t)-\partial \widehat \EE(q_2)(t)}{q_1(t)-q_2(t) }_{\hon}
\\&=\alpha \|(q_1-q_2)(t)\|_{\hon}^2+\dual{\kappa(q_1(t))-\kappa(q_2(t))}{q_1(t)-q_2(t)}_{H^1(\O)}
\\&\qquad \geq \alpha \|(q_1-q_2)(t)\|_{\hon}^2-L_\kappa\|q_1(t)-q_2(t)\|_{L^2(\O)}
{\|q_1(t)-q_2(t)\|}_{H^1(\O)}
\\&\qquad \geq(\underbrace{\alpha-L_\kappa}_{>0}) \|(q_1-q_2)(t)\|_{H^1(\O)}^2 \quad  \ae (0,T),
\quad \forall\,q_1, q_2 \in \hhyy.
\end{aligned}\end{equation}}Therefore, in the setting of \eqref{eq:n2} with $\HH=\mathbb{I}$ such a smallness assumption would force the uniform convexity of the ``energy" $\widehat \EE$. \end{remark}

Before we proceed, we observe that \eqref{eq:n2} fits in the setting of the previous sections. Indeed, if we set 
\begin{equation}\begin{aligned}\label{set}
X:=\lo,& \quad Y:=\hon, \quad \RR:=R, 
\end{aligned}\end{equation}
 then
Assumptions \ref{it:0}, \ref{it:1} and \ref{it:st1} are satisfied; here we note that $\dom \RR=\CC.$ 
We also observe that  the compatibility set from Definition \ref{def} is now given by 
\begin{equation}\label{def:l}
\ro{\mathfrak{L}(y_0):=\{\ell \in C([0,T];\hoon):-\UU'(0)+ \ell(0)-\kappa(y_0) \in \CC^\circ \}.}
\end{equation}
This is due to \[ \partial_2 R(y_0,0) =\{\zeta \in \hoon:\dual{\zeta}{\eta}_{\hon} \leq R(y_0,\eta) \quad \forall\,\eta \in \CC \},\]and \eqref{def:r}.

\ro{In order to be able to apply Theorem \ref{ex} for our  concrete setting, we  need to formulate some requirements on the general energy $\UU$ that guarantee that the assertions in Assumptions \ref{assu:ex} and \ref{assu:r} are true. These are given by the following
\begin{assumption}\label{assu:exd}
For each $M>0$, there exists $L(M)>0$ so that
\[\|\UU'(y_1)- \UU'(y_2)\|_{\hoon}\leq L(M)\|y_1-y_2\|_{\hon}\]for  $y_i \in \hon$ with $\|y_i\|\leq M, i=1,2.$ Moreover, there exist constants $c,c_0>0$ so that
\[\|\UU'(y)\|_{\hoon}\leq c\|y\|_{\hon}+c_0\quad \forall\, y \in \hon.\]
\end{assumption} 
\begin{assumption}\label{assu:rd}
Let $y_m \weakly y$ in $H^1(0,T;\hon)$, where $\{y_m\}  \subset  \YY:=\{y \in H^1_0(0,T;\hon):\dot y (t) \in \CC \text{ f.a.a.\,}t \in (0,T)\}$. Then,
for each $v \in  L^2(0,T;\CC )$ there exists a sequence $\{v_m\} \subset  L^2(0,T;\CC )$ with  $v_m \weakly v$ in $L^2(0,T;\hon)$ and 
\begin{align*}\limsup_{m \to \infty}  \int_0^T \UU'(y_m(t))(v_m(t)) \,dt  \leq  \int_0^T \UU'(y(t))(v(t)) \,dt.\end{align*}
\end{assumption}
\begin{remark}Under Assumption \ref{assu:exd}, the requirement in  Assumption \ref{assu:rd} is satisfied, provided that $\UU': W \to \hoon$ is continuous, where $W$ is a Banach space for which  the compact embedding $\hon \embed \embed W$ is true.\end{remark}We are now in the position to prove the following}
\begin{theorem}[Existence of solutions for \eqref{eq:n2}]\label{ex1}
\ro{Suppose that Assumptions \ref{assu:exd} and \ref{assu:rd} are satisfied.} Then, for each
 $\ell \in H^{1}(0,T;\hoon) \cap \ro{\mathfrak{L}(y_0)}$, the \ro{rate-independent} system \eqref{eq:n2} admits  solutions $q \in H^1_0(0,T;\hon)$. 
\end{theorem}

\begin{proof}\ro{In order to be able to apply Theorem \ref{ex} for our  concrete setting, we  need to verify the assertions in Assumptions \ref{assu:ex} and \ref{assu:r}. 
\\(i) \textit{Assumption \ref{assu:exd} implies Assumption \ref{assu:ex}}. Let $\e>0$ be arbitrary but fixed. We employ a truncation  argument  inspired from the proof of \cite[Prop.\,3.1]{kms}. For $M>0$ fixed we introduce the ODE in  Hilbert space 
\begin{equation}\label{eq:m}
 \dot  z(t)   =\frac{1}{\e} \V^{-1}_{\hon} (\mathbb{I}-P_{\partial_2 R(\HH(z)(t),0)})(-\UU'(P_M(z(t)))+\ell(t) ) \ \  \ae (0,T), \quad z(0)=0,
 \end{equation}where $\V_{\hon}$ is the Riesz isomorphism \ro{associated with }$\hon$. For each $\zeta \in \lo$, the operator $P_{\partial_2 \RR(\zeta, 0)}:\hoon \to \hoon $ has the same meaning as in Lemma \ref{lem:proj} and 
$P_{M}:\hon \to \hon$ denotes the projection on the closed ball $\clos{B_{\hon}(0,M)}.$ Thanks to Assumption \ref{assu:exd}, the mapping $\UU'\circ P_M:\hon \to \hoon$ is globally Lipschitzian, and thus, we can apply  \cite[Thm 8]{aos} to deduce that \eqref{eq:m} admits a unique solution $z_M \in H^{1}_0(0,T;\hon)$. Due to the growth condition in Assumption \ref{assu:exd} and by applying classical arguments involving Gronwall's lemma (see the proof of \cite[Thm 8]{aos}), one can show that 
\[\|z_M\|_{C([0,T];\hon)}\leq C \quad \forall\, M>0,\]where $C>0$ is independent of $M$ and depends only on the given data. Going back to \eqref{eq:m}, where we set  $M:=C+1$, then implies that \begin{equation}
 \dot  z(t)   =\frac{1}{\e} \V^{-1}_{\hon} (\mathbb{I}-P_{\partial_2 R(\HH(z)(t),0)})(-\UU'(z(t))+\ell(t) ) \ \  \ae (0,T), \quad z(0)=0
 \end{equation}admits a unique solution $z \in H^{1}(0,T;\hon)$. The desired assertion follows from Lemma \ref{lem:proj} and \eqref{deriv}.}
\\(ii) \textit{Assumption \ref{assu:rd} implies Assumption \ref{assu:r}.}
Let \begin{equation}\label{ym}
y_m \weakly y \quad \text{ in }H_0^1(0,T;\hon), 
\end{equation}where $\{y_m\}  \subset  \YY:=\{y \in H_0^1(0,T;\hon):\dot y (t) \in \CC \text{ f.a.a.\,}t \in (0,T)\}$. 
 {As a result of \eqref{ym},  {the compact embedding $\hon \embed \embed \lo$ and since $B(t,\cdot):\lo \to \lo$ is continuous}, cf.\,Assumption \ref{assu:c}.\ref{it:c2}, we have
\[B(t-s,y_m(s))-B(t-s, y(s)) \to 0 \quad \text{ in }\lo,\quad \forall\, s,t \in [0,T] \text{ with }s \leq t\]as $m \to \infty.$
Moreover, as $B(t,\cdot):\hon \to \lo$ is Lipschitzian, it holds
\[\|B(t-s,y_m(s))-B(t-s, y(s))\|_{\lo} \leq b(t-s)\|y_m(s)-y(s)\|_{\hon}, \]
for all $m\in \N$ and all $s,t \in [0,T]$ with $s \leq t$. On account of $b\in L^2(0,T),$ see Assumption \ref{assu:c}.\ref{it:c2}, and \eqref{ym} combined with the embedding $H_0^1(0,T;\hon) \embed L^\infty(0,T;\hon)$, one can employ Lebesgue's dominated convergence theorem, which gives in turn
\begin{equation}\begin{aligned}
&\int_0^T \|\HH(y_m)(t) - \HH(y) (t)\|^2_{\lo} \,dt 
\\&\qquad \leq \int_0^T  {\int_0^t \|B(t-s,y_m(s))-B(t-s, y(s))\|^2_{\lo}} \,ds\,dt
\ \to 0 \quad \text{as }m \to \infty.
\end{aligned}\end{equation}}
By 
Assumption \ref{assu:c}.\ref{it:c1}, we further deduce
 \begin{equation}\label{eq:k_l}
 \kappa(\HH(y_m)) \to \kappa(\HH(y)) \quad \text{in }L^2(0,T;\lo) \quad \text{as }m \to \infty,
 \end{equation}
 which together with \eqref{ym} yields 
 \begin{equation}\label{eq:conc0}\begin{aligned}
\int_0^T \int_\O \kappa(\HH(y_m)(t))\dot y_m (t)\,dx \,  d t \to \int_0^T \int_\O \kappa(\HH(y)(t))\dot y (t)\,dx  \, d t \quad \text{as }m \to \infty. \end{aligned}\end{equation}Since $\{y_m\}  \subset  \YY$ and $y \in \YY$, cf.\,Remark \ref{rem:rr},  \eqref{eq:conc0}
is equivalent to
 \begin{equation}\begin{aligned}
\lim_{m \to \infty}  \int_0^T R(\HH(y_m)(t),\dot y_m(t))  \,dt  = \int_0^T R(\HH(y)(t),\dot y(t))  \,dt ,\end{aligned}\end{equation}cf.\,\eqref{def:r}. Thus, the convergence in Assumption \ref{assu:r}.\ref{m1} is true. \ro{To check  Assumption \ref{assu:r}.\ref{m2}, let $v\in L^2(0,T;\CC)$ be arbitrary but fixed and 
denote by $\{v_m\}$ the associated recovery sequence from Assumption \ref{assu:rd}. Then, \eqref{eq:k_l}, \eqref{def:r} and Assumption \ref{assu:rd} yield the assertion in Assumption \ref{assu:r}.\ref{m2}. Hence, our present setting satisfies the entire Assumption \ref{assu:r}.} Now, applying Theorem \ref{ex} gives the desired result.
\end{proof}


Since Assumptions \ref{assu:exd} and \ref{assu:rd} imply that Assumptions \ref{assu:ex} and  \ref{assu:r} are satisfied in the present framework, see the proof of Theorem \ref{ex1}, we can apply Theorem \ref{ex_opt} to obtain existence of optimal controls for minimization problems such as \eqref{eq:min}, which are governed by \eqref{eq:n2}. By setting $Z:=\lo \embed \embed \hon$, we arrive at  the following 

\begin{theorem}\label{ex_opt2}
Let $j:H^1_0(0,T;\hon) \to \R$ be a weakly lower semicontinuous mapping, bounded from below. \ro{Under Assumptions \ref{assu:exd} and \ref{assu:rd},} the optimization problem \begin{equation}\label{eq:min2}
 \left.
 \begin{aligned}
  \min \quad & j(q)+\frac{1}{2}\|\ell\|^2_{H^1(0,T;\lo)}\\
  \text{s.t.} \quad & (\ell,q) \in {H^1(0,T;\lo)}\times {H_0^1(0,T;\hon)},
  \\\quad & q \text{ solves }
\eqref{eq:n2} \text{ with right hand side\ }\ell \ro{\in \mathfrak{L}(y_0)}
 \end{aligned}
 \quad \right\}
\end{equation}admits at least one solution.
\end{theorem}

\subsection{Lipschitz continuity of the solution map to \eqref{eq:n2}}\label{u}
This section is dedicated to the discussion of uniqueness of solutions for \ro{history-dependent} \ro{rate-independent} systems of the type \eqref{eq:n2}. To ensure the viability of the existence result from Theorem \ref{ex1}, \ro{Assumptions \ref{assu:c},  \ref{assu:exd}  and \ref{assu:rd}} are tacitly assumed in all what follows, without mentioning them every time.


We underline that Proposition \ref{lem:ess} below is the key to the sensitivity analysis in terms of Lipschitz estimates for the solution map \ro{associated with }\eqref{eq:n2}, see Theorem \ref{thm:lip} below. Proposition \ref{lem:ess} not only  allows us to 
prove the unique solvability of \eqref{eq:n2}, but it is also the essential ingredient \ro{for} showing uniform
Lipschitz continuity of the \textit{viscous} solution map, cf.\,Proposition \ref{lem:lip_c} below. These two results will be crucial for the derivation of optimality conditions for the control of \eqref{eq:n2} in a follow-up paper; see Remark \ref{rem:fw} below. 

For the sake of a better overview in the proof of Proposition \ref{lem:ess}, we will work from now on with the  \textit{history operator} 
$\HH:L^1(0,T;H^1(\O)) \to W^{1,1}(0,T;\hon)$ defined as \begin{equation}\label{h_op}
 [0,T] \ni t \mapsto \HH(y)(t):=\int_0^t  y(s) \,ds +y_0 \in \hon,\end{equation}
where  $y_0\in \hon.$ The general case from Assumption \ref{assu:c}.\ref{it:c2} can be considered as well, provided that $\HH$ is linear, see Remark \ref{rem:h} below.

For our upcoming analysis, we need to impose a {smoothness} assumption on the mapping acting on the first argument of $R$.
{
\begin{assumption}\label{assu:k}The mapping $\kappa$ is differentiable with $\kappa' \in W^{1,1}(\R).$ Moreover, there exist constants $c,c_0>0$ so that 
\[|\kappa''(y)| \leq c|y|^{\varrho}+c_0 \quad \forall\,y \in \R,\]where  $\varrho$ is a fixed value that satisfies $\varrho\in [1,\infty)$ for $N=2$ and $\varrho\in [1,3]$ for $N=3$.
\end{assumption}}
{Assumption \ref{assu:k} combined with the embedding $L^\infty(0,T;\hon) \embed L^{2\varrho}((0,T)\times \O)$ implies that there exist constants $C,C_0>0$, dependent only on the given data, so that 
\begin{equation}\label{k''}
\|\kappa''(y)\|_{L^2((0,T) \times \O)} \leq C\|y\|^\varrho_{L^\infty((0,T);\hon)} +C_0 \quad \forall\,y \in L^\infty(0,T;\hon).\end{equation}This will be employed in the proof of Proposition \ref{lem:ess} below, see \eqref{eq:k}. Note that Assumption \ref{assu:k} is also used for the proof of \eqref{ff}, in the case where the argument of $\kappa''$ is a not necessarily measurable mapping.}

\ro{For convenience of the reader, let us recall here that the fixed parameter $\alpha>0$ was introduced in Assumption \ref{assu:c}. We are now in the position to prove the following}
\begin{proposition}[Essential estimate]\label{lem:ess}
Let Assumption \ref{assu:k} be satisfied. For given $M>0$ and 
 $q_1,q_2 \in H^1(0,T;\hon)$ with $\|q_i\|_{{H^1(0,T;{\hon})}} \leq M, i=1,2,$ it holds
 \begin{equation*}\begin{aligned}
 \int_0^t-&\dual{\F(q_2)(s)-\F(q_1)(s)}{\dot q_2(s)-\dot q_1(s)}_{\hon}\,ds
 \\& \leq  \frac{\alpha}{4}\|(q_2- q_1)(t)\|_{\lo}^2+c\,\int_0^t \| (q_2- q_1)(s)\|_{L^6(\O)}^2\,ds \quad  \forall\, t \in [0,T],\end{aligned}\end{equation*}  \end{proposition} where $c>0$ is  a constant dependent only on {$L_\kappa,\alpha,T,\varrho $ and $M$}.

 \begin{proof}
 Throughout the proof $t \in [0,T]$ is arbitrary, but  fixed. We  abbreviate
\[E(s):=-\dual{\F(q_2)(s)-\F(q_1)(s)}{\dot q_2(s)-\dot q_1(s)}_{\hon}, \quad s\in[0,t].\]
We begin by observing that 
\[\F(q_2)-\F(q_1)=\kappa'(\HH(q_1))(\HH(q_2)-\HH(q_1))+g\quad \ae (0,T) \times \O,\]
where we abbreviate 
 \begin{equation}\label{def:g}\begin{aligned}
 g:={\kappa''(\HH(q_\gamma))(\HH(q_2)-\HH(q_1))^2}
  \end{aligned}\end{equation}
 and 
$\HH(q_\gamma)(s,x):=\gamma(s,x)(\HH(q_2)-\HH(q_1))(s,x)+\HH(q_1)(s,x),$ $\gamma(s,x)\in [0,1], (s,x) \in (0,T) \times \O.$


Thus,
\begin{equation}\label{E}\begin{aligned}
\int_0^t E(s) \,ds
&=-\int_0^t \int_\O (\F(q_2)-\F(q_1))(s,x)\partial_s (q_2- q_1)(s,x) \,dx\,ds
\\&=-\int_0^t \int_\O \kappa'(\HH(q_1))(\HH(q_2)-\HH(q_1))(s,x)\partial_s (q_2- q_1)(s,x)\,dx \,ds
\\&\quad -\int_0^t \int_\O  g(s,x)\partial_s (q_2- q_1)(s,x)\,dx \,ds
\\& =:-E_1-E_2.
\end{aligned}\end{equation}
Note that,  we denote by $\partial_s (q_2- q_1) \in L^2((0,T)\times \O)$  the partial weak derivative with respect to\ time of $(q_2- q_1)\in H^1((0,T)\times \O)$, cf.\,\cite[Lem.\,10.1,\,p.\,137]{friedman}.
In light of \eqref{h_op}, we can rewrite 
\begin{equation}\begin{aligned}\label{e11}
E_1=
\int_0^t \int_\O \Big( \int_0^s (q_2-q_1)(\zeta,x)\,d \zeta\Big) \kappa'(\HH(q_1)(s,x)) \partial_s (q_2- q_1)(s,x)\,dx \,ds.
\end{aligned}\end{equation}
Since $q_1 \in H^1(0,T;\hon)$ by assumption, $\HH(q_1) \in H^2(0,T;\hon) \subset  H^1((0,T)\times \O)$, as a consequence of \eqref{h_op} and \cite[Lem.\,10.1,\,p.\,137]{friedman}. According to \cite{mm_serrin}, the mapping $\kappa'(\HH(q_1))$ belongs to $W^{1,1}((0,T)\times \O)$, in view  of {\eqref{k''}}; note that this is true \ro{even though} $\kappa'$ is not Lipschitz continuous, since $\kappa'$ is defined on an one-dimensional space, cf.\,the introduction in \cite{mm_serrin} and \cite[p.\,202]{mm}.
 Hence, the following integration by parts formula is true 
\begin{equation*}\begin{aligned}
&\kappa'(\HH(q_1))(s,x) \partial_s (q_2- q_1)(s,x)
\\& =\partial_s[\kappa'(\HH(q_1))  (q_2- q_1)](s,x)-\partial_s \kappa'(\HH(q_1))(s,x)  (q_2- q_1)(s,x) \quad \ae (0,T) \times \O.
\end{aligned}\end{equation*}
Introducing the above  identity in \eqref{e11} yields
\begin{equation}\label{E1}\begin{aligned}
E_1&=
\int_0^t \int_\O \Big( \int_0^s (q_2-q_1)(\zeta,x)\,d \zeta\Big) \partial_s[\kappa'(\HH(q_1))  (q_2- q_1)](s,x)\,dx \,ds
\\&\quad -\int_0^t \int_\O \Big( \int_0^s (q_2-q_1)(\zeta,x)\,d \zeta\Big) \partial_s \kappa'(\HH(q_1))(s,x)  (q_2- q_1)(s,x)\,dx \,ds
\\\quad &:=e_1+e_2.
\end{aligned}\end{equation}
By employing Fubini's theorem one obtains
\begin{equation*}\begin{aligned}
e_1&=
 \int_0^t \int_\O \Big( \int_s^t \partial_s[\kappa'(\HH(q_1))  (q_2- q_1)](\zeta,x)\,d \zeta \Big)  (q_2- q_1)(s,x)\,dx \,ds
\\ \quad &= \int_0^t \int_\O [\kappa'(\HH(q_1))  (q_2- q_1)(t,x)-\kappa'(\HH(q_1))  (q_2- q_1)(s,x)]  (q_2- q_1)(s,x)\,dx \,ds.
\end{aligned}\end{equation*}
In light of H\"older's inequality, this leads to 
\[|e_1| \leq \int_0^t L_\kappa\,  (\|(q_2- q_1)(t)\|_{\lo}\| (q_2- q_1)(s)\|_{\lo} +\| (q_2- q_1)(s)\|_{\lo}^2)\,ds,\] where we also relied on the Lipschitz continuity of $\kappa$, cf.\,Assumption \ref{assu:c}.\ref{it:c1}.
To get a final estimate for $e_1$ we apply 
 Young's inequality, which gives in turn
\begin{equation}\label{e1}
|e_1| \leq \frac{\alpha}{4}\|(q_2- q_1)(t)\|_{\lo}^2+c(L_\kappa,T,\alpha)\,\int_0^t \| (q_2- q_1)(s)\|_{\lo}^2\,ds,
\end{equation}where $c(L_\kappa,T,\alpha)>0$ does not depend on any parameters other than  $L_\kappa,T$ and $\alpha$.
 
 To obtain a suitable estimate for $e_2$, we first observe that 
\[\partial_s \kappa'(\HH(q_1))(s,x)=\kappa''(\HH(q_1)(s,x))(q_1(s,x)) \quad \ae (0,T) \times \O,\] in view of \eqref{h_op} and chain rule for composition of Sobolev functions, see e.g.\,\cite{mm_serrin}. 
%
Since {$\|q_1\|_{L^\infty(0,T;\hon)} \leq M$, by assumption,} one has 
\begin{equation}\label{eq:k}
\| \partial_s \kappa'(\HH(q_1))\|_{L^2(0,T;L^{3/2}(\O))} \leq \|\kappa''(\HH(q_1))\|_{L^2((0,T)\times \O)}\,c\,M \leq {C(M,\varrho,T)},\end{equation}where we relied again on {\eqref{k''}}.
  Going back to $e_2$, we now see, by applying H\"older's inequality, that 
 \begin{equation*}\begin{aligned}
| e_2|&=\Big|\int_0^t \int_\O \Big( \int_0^s (q_2-q_1)(\zeta,x)\,d \zeta\Big) \partial_s \kappa'(\HH(q_1))(s,x)  (q_2- q_1)(s,x)\,dx \,ds\Big|
\\ \quad &\leq \int_0^t \Big\| \int_0^s (q_2-q_1)(\zeta)\,d \zeta\Big\|_{L^6(\O)}\| \partial_s \kappa'(\HH(q_1))(s)\|_{L^{3/2}(\O)}  \|(q_2- q_1)(s)\|_{L^6(\O)} \,ds
\\ \quad &\leq  \int_0^t \|(q_2-q_1)(s)\|_{L^6(\O)}\,d s \int_0^t \| \partial_s \kappa'(\HH(q_1))(s)\|_{L^{3/2}(\O)}  \|(q_2- q_1)(s)\|_{L^6(\O)} \,ds
\\ \quad &\leq  \| \partial_s \kappa'(\HH(q_1))\|_{L^2(0,T;L^{3/2}(\O))} \|q_2-q_1\|_{L^2(0,t;L^6(\O))}\|q_2-q_1\|_{L^1(0,t;L^6(\O))}.
\end{aligned}\end{equation*}
Thus,
\begin{equation}\label{e2}
|e_2 | \leq {c( M,\varrho,T)}\int_0^t   \|(q_2- q_1)(s)\|^2_{L^6(\O)} \,ds,
\end{equation}on account of \eqref{eq:k}; note that ${c( M,\varrho,T)}>0$ depends only on {$M,\varrho$ and $T$}.

Finally, we need to estimate the second term appearing in \eqref{E}. To do so, we recall the definition of $g$ from \eqref{def:g}:
  \begin{equation*}\begin{aligned}
g ={\kappa''(\HH(q_\gamma))(\HH(q_2)-\HH(q_1))^2},
\end{aligned}\end{equation*}where $\HH(q_\gamma)(s,x)=\gamma(s,x)(\HH(q_2)-\HH(q_1))(s,x)+\HH(q_1)(s,x),$ $\gamma(s,x)\in [0,1], (s,x) \in (0,T) \times \O.$
Since $g$ is a measurable function, the same is true for the mapping\[f:(s,x) \mapsto \raisebox{3pt}{$\chi$}_{\HH(q_2) \neq \HH(q_1)} (s,x){\kappa''(\HH(q_\gamma)(s,x))}.\] 
{Thanks to Assumption \ref{assu:k}, it holds
\begin{align*}
|\kappa''(\HH(q_\gamma)(s,x))| &\leq c|\HH(q_\gamma)(s,x)|^{\varrho}+c_0 
\\& \leq c(\varrho)(|\HH(q_2)-\HH(q_1))(s,x)|^\varrho+|\HH(q_1)(s,x)|^\varrho)+c_0 ,\end{align*}whence
\begin{equation}\label{ff}
\|f\|^2_{L^2((0,T) \times \O)}\leq C(\varrho,T)\,\max\{\|\HH(q_1)\|^{2\varrho}_{L^{2\varrho}((0,T) \times \O)},\|\HH(q_2)\|^{2\varrho}_{L^{2\varrho}((0,T) \times \O)}\}+c_0 \end{equation}follows.
As  $\|q_i\|_{L^\infty(0,T;\hon)} \leq M, i=1,2$, by assumption, and since the embedding $L^\infty(0,T;\hon)\embed L^{2\varrho}((0,T) \times \O)$ is true, we deduce that
\[\|f\|_{L^2((0,T) \times \O)}\leq C(\varrho,T,M) .\]
Further, we  estimate as follows
\begin{equation}\label{g}\begin{aligned}
\|g\|_{L^2(0,t;L^{\frac{6}{5}}(\O))} &\leq \|f\|_{L^2((0,T) \times \O)}\| \HH(q_2)-\HH(q_1)\|^2_{L^\infty(0,t;L^6(\O))}
\\&\leq 
C(\varrho,T,M) \| \HH(q_2)-\HH(q_1)\|^2_{L^\infty(0,t;L^6(\O))}. \end{aligned} \end{equation}} Thus, by employing again H\"older's inequality, we have
\begin{equation*}\begin{aligned}
|E_2|&\leq \int_0^t \Big|\int_\O g(s,x)\partial_s (q_2- q_1)(s,x)\,dx\Big| \,ds 
\\\quad &\leq \|g\|_{L^2(0,t;L^{\frac{6}{5}}(\O))} \|\dot q_2-\dot q_1\|_{L^2(0,t;L^{6}(\O))} 
\\\quad &\leq { C(\varrho,T,M)} \| \HH(q_2)-\HH(q_1)\|^2_{L^\infty(0,t;L^6(\O))}  \|\dot q_2-\dot q_1\|_{L^2(0,T;L^{6}(\O))},
\end{aligned}\end{equation*}in view of \eqref{g}.
Since  $\|q_i\|_{H^1(0,T;L^{6}(\O))} \leq M, i=1,2,$ by assumption,  the above estimate can be continued as 
\begin{equation}\label{E2}\begin{aligned}
|E_2|&\leq 
 {C(\varrho,T,M)} \,\Big(\int_0^t \| (q_2- q_1)(s)\|_{L^{6}(\O)}\,ds \Big)^2
   \\\quad &\leq {C(\varrho,T,M)} \,\int_0^t \| (q_2- q_1)(s)\|_{L^{6}(\O)}^2 \,ds,
   \end{aligned}\end{equation}
where we relied on  the definition of $\HH,$ cf.\,\eqref{h_op}.

   Summarizing, we obtained from \eqref{E}, 
   \eqref{E1}, \eqref{e1}, \eqref{e2} and \eqref{E2}, that 
   \begin{equation*}\begin{aligned}
\int_0^t E(s) \,ds
&\leq |e_1|+|e_2|+|E_2|
\\\quad &\leq  \frac{\alpha}{4}\|(q_2- q_1)(t)\|_{\lo}^2+c(L_\kappa, T, \alpha)\int_0^t \| (q_2- q_1)(s)\|_{\lo}^2\,ds
\\ &\quad
+C(\varrho,T,M)\,\int_0^t \| (q_2- q_1)(s)\|_{L^{6}(\O)}^2 \,ds 
\\\quad &\leq  \frac{\alpha}{4}\|(q_2- q_1)(t)\|_{\lo}^2+c\,\int_0^t \| (q_2- q_1)(s)\|_{L^6(\O)}^2\,ds \quad \forall\, t \in [0,T],
\end{aligned}\end{equation*}where $c>0$ is  a constant dependent only on $L_\kappa,\alpha,T$ {and $M$}. The proof is now complete.

\end{proof}

\ro{One important consequence of the estimate from Proposition \ref{lem:ess} is a Lipschitz condition for the (for now) multivalued solution operator of \eqref{eq:n2}, which gives in turn the uniqueness. To be able to prove this, we need to impose a local Lipschitz condition on the second derivative of $\UU$. This type of requirement is also used in \cite[App.\,B]{kms}, see \cite[Eq.\,(5.2)]{kms}, to deduce uniqueness of solutions in the uniformly convex (finite-dimensional) case. 
From now on we assume that
\begin{assumption}\label{u''}
For each $M>0$, there exists $L(M)>0$ so that
\[\|\UU''(y_1)- \UU''(y_2)\|_{\LL(\hon,\hoon)}\leq L(M)\|y_1-y_2\|_{\hon}\]for  $y_i \in \hon$ with $\|y_i\|\leq M, i=1,2.$\end{assumption} }

  \begin{theorem}[Lipschitz stability, Uniqueness]\label{thm:lip}
  \ro{Suppose that Assumptions \ref{assu:exd}, \ref{assu:rd}, \ref{assu:k} and \ref{u''} are satisfied.} 
Let  $\ell_1,\ell_2 \in H^1(0,T;\hoon) \cap \ro{\mathfrak{L}(y_0)}$ be given such that   $\|\ell_i\|_{H^1(0,T;\hoon)} \leq \widehat M, i=1,2,$ for some fixed $\widehat M>0$. Then 
 \[\|q_1-q_2\|_{C([0,T];\hon)} \leq c(\widehat M)\,\|\ell_1-\ell_2\|_{\ro{W^{1,1}}(0,T;\hoon)},\]where $q_i \in H_0^1(0,T;\hon)$ is a solution to \eqref{eq:n2} with right hand side $\ell_i, i=1,2,$ and $c(\widehat M)>0$ is  a constant dependent only on $\widehat M$ and the given fixed parameters. Thus, the \ro{history-dependent} evolution \eqref{eq:n2} admits a unique solution in $H_0^1(0,T;\hon)$. 
 \end{theorem} 
\begin{remark}\label{opt_lip}
It is not to be expected that the Lipschitz estimate in Theorem \ref{thm:lip} holds true in better spaces, cf.\,\cite[Lem.\,3.2]{sww}, where a \ro{rate-independent} system \ro{with quadratic energy and } without a history variable is considered (\ro{state-independent} dissipation potential). As opposed to \cite{sww}, we require that the right-hand sides $\ell$ in \eqref{eq:n2} are uniformly bounded. We do so,  in order to be able to apply the essential estimate in Proposition \ref{lem:ess} concerning the mapping $\kappa$ that acts on the history term. Since an additional state dependence is not taken into account in \cite{sww}, a  boundedness condition for the loads is not needed there.
\end{remark}
\begin{proof}[Proof of Theorem \ref{thm:lip}]
Since $q_i$ solves \eqref{eq:n2} with right hand side\ $\ell_i, i=1,2$, we have 
\begin{equation}\label{eq:evi21}
0  \geq \dual{\ro{-\UU'( q_1(t))} +\ell_1(t)-\F(q_1)(t)}{\eta-\dot q_1(t)}_{\hon} \quad \forall\, \eta \in \CC, 
\end{equation}
\begin{equation}\label{eq:evi22}
0  \geq \dual{\ro{-\UU'( q_2(t))} +\ell_2(t)-\F(q_2)(t)}{\eta-\dot q_2(t)}_{\hon} \quad \forall\, \eta \in \CC
\end{equation}a.e.\ in $(0,T)$, see \eqref{def:e1} and \eqref{def:r}.
Testing \eqref{eq:evi21} and \eqref{eq:evi22} with $\dot q_2$ and $\dot q_1$, respectively, yields 
\begin{equation*}\begin{aligned}
0 &\geq \ro{\dual{\UU'( q_2(t))-\UU'( q_1(t))}{\dot q_2(t)-\dot q_1(t)}_{\hon}}
+\dual{\ell_1(t)-\ell_2(t)}{\dot q_2(t)-\dot q_1(t)}_{\hon}
\\&+\dual{\F(q_2)(t)-\F(q_1)(t)}{\dot q_2(t)-\dot q_1(t)}_{\hon} \quad \ae (0,T).
\end{aligned}\end{equation*}
\ro{Thus, for all $t\in [0,T]$ it holds 
\begin{equation}\label{sigma_t}\begin{aligned}
 \int_0^t \tilde \sigma(s)\,ds
&  \leq -\int_0^t \dual{\dot \ell_1(s)-\dot \ell_2(s)}{ q_1(s)- q_2(s)}_{\hon} \,ds
\\&\quad +\dual{\ell_1(t)-\ell_2(t)}{ q_1(t)- q_2(t)}_{\hon}+\int_0^t E(s)\,ds,
 \end{aligned}\end{equation} where we abbreviate
 \[\tilde \sigma (s):= \ro{\dual{\UU'( q_2(s))-\UU'( q_1(s))}{\dot q_2(s)-\dot q_1(s)}_{\hon}} ,\]
 \[E(s):=-\dual{\F(q_2)(s)-\F(q_1)(s)}{\dot q_2(s)-\dot q_1(s)}_{\hon}, \quad s\in[0,t].\]
 From Lemma \ref{lem:est_ris} we infer that 
 \begin{equation}\label{b_q}
 {\|q_i\|_{H^1(0,T;\hon)} \leq c\,\|\ell_i\|_{H^1(0,T;\hoon)}+c_0\leq c\widehat M+c_0,\ i=1,2.}\end{equation} In light of Assumption \ref{assu:k}, this allows us to  apply 
  Proposition \ref{lem:ess} with \[M:=c\,\widehat M{+c_0}.\]Note that, according to Lemma \ref{lem:est_ris}, $c>0$ is independent of $\ell_i,\ i=1,2.$
  Hence, \eqref{sigma_t} can be continued as 
  \begin{equation}\label{sigma_tt}
  \begin{aligned}
   \int_0^t \tilde \sigma(s)\,ds &\leq 
 \|\dot \ell_1-\dot \ell_2\|_{L^1(0,T;\hoon)} \|q_2- q_1\|_{C([0,T];\hon)}
\\&\qquad +\| \ell_1-\ell_2\|_{C([0,T];\hoon)}\|q_2- q_1\|_{C([0,T];\hon)}
\\&\qquad +  \frac{\alpha}{4}\|(q_2- q_1)(t)\|_{\lo}^2
 \\&\qquad +c(L_\kappa,\alpha,T, \widehat M, {\varrho})\int_0^t \| (q_2- q_1)(s)\|_{L^6(\O)}^2\,ds 
 \\&\quad\leq C\| \ell_1-\ell_2\|_{W^{1,1}(0,T;\hoon)} \|q_2- q_1\|_{C([0,T];\hon)}
 \\&\qquad +  \frac{\alpha}{4}\|(q_2- q_1)(t)\|_{\lo}^2
 \\&\qquad +C(L_\kappa,\alpha,T, \widehat M, {\varrho},\O) \int_0^t \| (q_2- q_1)(s)\|_{\hon}^2\,ds \quad \forall\,t\in [0,T],
  \end{aligned}\end{equation}where we relied on the embedding $\hon \embed L^{6}(\O)$. Note that all the constants appearing in \eqref{sigma_tt} depend only on the given data.
 Further, we use an argument which  is inspired by the proofs of \cite[Thm.\,7.4]{mt2004}, see also \cite[App.\,B]{kms}. To this end, we define
 \[\sigma(t):=\dual{\UU'( q_2(t))-\UU'( q_1(t))}{ q_2(t)- q_1(t)}_{\hon}\]and observe that
   \[\sigma '(t)=\UU''( q_2(t))[\dot q_2(t),(q_2(t)- q_1(t))]-\UU''( q_1(t))[\dot q_1(t),( q_2(t)-q_1(t))]+\tilde \sigma (t)\]a.e.\,in $(0,T).$ 
  Rearranging the terms yields
    \begin{align*}\sigma '(t)
    =
    \UU''( q_2(t))[(q_2(t)- q_1(t)),\dot q_2(t)]+
   \dual{\UU'( q_1(t))-\UU'( q_2(t))}{\dot q_2(t)}_{\hon}
  \\  -\UU''( q_1(t))[( q_2(t)-q_1(t)),\dot q_1(t)]
    +   \dual{\UU'( q_2(t))-\UU'( q_1(t))}{\dot q_1(t)}_{\hon}
 +2\tilde \sigma(t),
    \end{align*}where we recall
     \[\tilde \sigma (t)= \ro{\dual{\UU'( q_2(t))-\UU'( q_1(t))}{\dot q_2(t)-\dot q_1(t)}_{\hon}} .\]
 Thanks to \eqref{b_q} combined with $H^1(0,T;\hon) \embed L^\infty(0,T;\hon)$, we can make use of the local Lipschitz continuity of $\UU''$ from Assumption \ref{u''}, which gives in turn
      \begin{equation}\label{sigma'}
      \sigma '(t)\leq  L(c,\widehat M,c_0,T)\|q_2(t)- q_1(t)\|^2_{\hon}(\|\dot q_2(t)\|_{\hon}+\|\dot q_1(t)\|_{\hon})
 +2\tilde \sigma(t)
\end{equation}
    a.e.\,in $(0,T).$ 
    Further, by  the uniform convexity of $\UU$ and \eqref{sigma_tt}, we have for all $t\in [0,T]$ the following 
        \begin{align*}
   \alpha \|q_1(t)-q_2(t)\|_{\hon}^2&\leq     \sigma(t) = \int_0^t \sigma'(s)\,ds
        \\&\leq 2C\| \ell_1-\ell_2\|_{W^{1,1}(0,T;\hoon)} \|q_2- q_1\|_{C([0,T];\hon)}
 \\&\qquad +  \frac{\alpha}{2}\|(q_2- q_1)(t)\|_{\lo}^2
 \\&\quad +
\int_0^t [2  C+L(\|\dot q_2(s)\|+\|\dot q_1(s)\|)] \| q_2(s)- q_1(s)\|_{\hon}^2\,ds .
    \end{align*}
    We subtract $\frac{\alpha}{2}\|q_2(t)- q_1(t)\|^2_{\hon}
$ from both sides and apply Gronwall's inequality, which results in
    \begin{align*}
    &\|q_1(t)-q_2(t)\|_{\hon}^2
    \\& \leq C\| \ell_1-\ell_2\|_{W^{1,1}(0,T;\hoon)} \|q_2- q_1\|_{C([0,T];\hon)}\exp(c+\int_0^t \|\dot q_2(s)\|+\|\dot q_1(s)\| \,ds )
    \\&\leq C\| \ell_1-\ell_2\|_{W^{1,1}(0,T;\hoon)} \|q_2- q_1\|_{C([0,T];\hon)} \quad \forall\,t\in[0,T],
    \end{align*}where for the last estimate we employed again \eqref{b_q}. Note that the constant $C>0$ depends only on $\widehat M$ and other given parameters. The proof is now complete.}
\end{proof}

We end  this subsection with some comments concerning the {smoothness} requirement in Assumption \ref{assu:k} as well as regarding the  structure of the setting in the present section and how it influences the unique solvability result.
\ro{\begin{remark}\label{rem:u}
We underline that the assumptions concerning $\UU$ in Theorem \ref{thm:lip} are indispensable when one works with general energies such as \eqref{def:e1}; note that Assumptions \ref{assu:exd}, \ref{assu:rd} and \ref{u''} are automatically fulfilled when $\UU$ is quadratic. While Assumptions \ref{assu:exd} and \ref{assu:rd} ensure the existence of solutions in Theorem \ref{ex1}, Assumption  \ref{u''} is a  smoothness requirement that guarantees uniqueness. This is only used to show \eqref{sigma'} and, instead of a local Lipschitz condition for the second derivative, some authors ask for $C^3$ smoothness of the energy to prove \eqref{sigma'} \cite[Thm.\,7.4]{mt2004}. Let us point out that, for the  existence result in \cite{mie_ros}, a condition that translates to weak-to-weak continuity of $\UU'$ is required, namely \cite[Eq.\,(4.4)]{mie_ros}. It can be easily checked that this implies  Assumption \ref{assu:rd}. Finally, we recall that Assumption \ref{assu:exd} is needed just for the existence of  solutions for the viscous model associated with \eqref{eq:n2}, i.e., \eqref{eq:q11}. We think that this requirement may be avoided partially, that is, one can dispense off the global growth condition in Assumption \ref{assu:exd}, if one does not work with the viscous regularization, and instead, performs a time discretization  and passes to the limit in the time-discrete energy \cite{mie_ros}. See  Remark \ref{rem:visc} for details. In Remark \ref{rem:fw} below we explain why we choose the vanishing viscosity approach in the present paper.
\end{remark}}
\begin{remark}[Comparison with other contributions regarding Assumption \ref{assu:k}]\label{rem:non}
We point out that, in order to obtain uniqueness of solutions, related works \cite{BS,TR,mie_ros} impose, in addition to a smoothness condition on the dissipation functional or on the Minkowski functional, bounds from below \cite[Eq.\,(5.9),(5.10)]{mie_ros} and/or smallness assumptions.  
 Moreover, the crucial bound \cite[Eq.\,(5.12)]{mie_ros} is imposed, which in our case reads as follows: $\exists \,\delta>0$ so that 
\[\alpha \|v\|^2_{\hon}+\int_\O \kappa'(\zeta)v^2 \geq \delta \|v\|^2_{\hon} \quad \forall\,\zeta, v \in \hon.\] Clearly, this is not necessarily fulfilled in our setting as $\kappa$ is non-increasing in applications, see Remark \ref{non-conv}. 

  \end{remark}
 {The next two remarks concern the proof of Proposition \ref{lem:ess}.}
\begin{remark}[General history operator]\label{rem:h}
Our choice  to work with the simple Volterra operator in \eqref{h_op} is entirely motivated by the complexity of the proof of Proposition \ref{lem:ess}.
If one wants the \textit{history operator} $\HH$ to preserve the structure from Assumption \ref{assu:c}.\ref{it:c2},  one needs to assume that $\kappa'\in W^{1,p}(\R)$ with some $p\geq 2$ and that  $B\in W^{1,2}(0,T; \KK(\hon,\lo) \cap \LL(\hon,L^\nu(\O)))$  with $\nu \geq 2$ satisfying
\begin{subequations}\begin{gather}
\frac{2}{\theta}+\frac{1}{p}+\frac{1}{\nu}\leq 1,\label{nu1}
\\\frac{2}{\nu}+\frac{1}{p}+\frac{1}{\theta}\leq 1,\label{nu2}
\end{gather}\end{subequations}where $\theta \geq 2$ is such that $\hon \embed L^\theta(\O);$  {note that  by $\KK(\hon,\lo)$ we denote the subspace of compact operators of $\LL(\hon,\lo)$.}
Then, a close inspection of the proof of Proposition \ref{lem:ess} shows that the crucial estimate still holds true, provided that the conditions \eqref{nu1}-\eqref{nu2} are fulfilled. 
In particular, this shows that the smoothness of the derivative of $\kappa$ depends on the regularity of the space where $B$ takes its values at each time point and on the dimension of the setting. However, to ensure that the arguments in the proof of Proposition \ref{lem:ess} are not difficult to follow,  we choose to stick with the history operator \eqref{h_op}. \end{remark}

\begin{remark}[Linearity of $R$ w.r.t.\,its second argument]\label{rem:lin}
To show the uniqueness result in Theorem  \ref{thm:lip}, the linearity of $\RR$ (on $\dom \RR$) with respect to its second argument was essential. 
\\
If this is not available, say $\RR=\widehat R$ from \eqref{def:t0}, then we can no longer argue as in the proof of \eqref{E}, since 
 \begin{equation*}\begin{aligned}
 \int_0^t E(s)\,ds = \int_0^t-\dual{\F(q_2)(s)-\F(q_1)(s)}{|\dot q_2(s)|-|\dot q_1(s)|}_{\hon}\,ds.\end{aligned}\end{equation*}
That is, we cannot make use of the derivation by parts arguments employed in the proof of Proposition \ref{lem:ess} to obtain suitable bounds for the expression $\int_0^t E(s)\,ds$, see\,\eqref{e11}.
\end{remark}

\subsection{Consequences of the uniqueness result and future work}\label{sec:c}
We end the paper by mentioning some fundamental  consequences of Proposition \ref{lem:ess}.

We point out that the uniqueness result from Theorem \ref{thm:lip} opens the door to future research concerning the optimal control of \eqref{eq:n2}, see \eqref{eq:min2}, as it allows us to define a single-valued control-to-state operator. 
By revisiting the results shown in the previous sections, we conclude some important properties of the solution map to \eqref{eq:n2}.
\begin{corollary}\label{cor}
 \ro{Suppose that  Assumptions \ref{assu:exd}, \ref{assu:rd}, \ref{assu:k} and \ref{u''} are satisfied.} Then, the (single-valued) solution operator \[\SS:H^{1}_0(0,T;\hoon)  \to H_0^{1}(0,T;\hon)\] \ro{associated with }\eqref{eq:n2} is weakly continuous from $H^{1}_0(0,T;\lo)$  to $H_0^{1}(0,T;\hon)$. Moreover, $\SS$ is  locally Lipschitz continuous in the sense of Theorem \ref{thm:lip}. 
\end{corollary}\begin{proof}
Note that $ H_0^1(0,T;\hoon) \subset  \ro{\mathfrak{L}(y_0)},$ since $\kappa \geq 0$ by Assumption  \ref{assu:c}.\ref{it:c1}; hence, $\SS$ is well-defined, in view of Theorem \ref{ex1}.
The weak continuity result follows from Lemma \ref{lem:est_ris} and by employing the same arguments as in the proof of Theorem \ref{ex_opt}. The local Lipschitz continuity is a direct consequence of  Theorem \ref{thm:lip}.
\end{proof}

The statements in Corollary \ref{cor} will be useful in a follow-up paper. Therein, our main interest will be to derive an optimality system for the control of \eqref{eq:n2} by performing a limit analysis $\e \searrow 0$ in the optimality conditions \cite{aos} \ro{associated with }the control of the viscous problem \eqref{eq:q1}.  {In the finite-dimensional framework with quadratic energy, this technique has  been recently employed in \cite{bp}.}

We note that, adapted to our present setting, the viscous differential inclusion \ro{associated with }\eqref{eq:n2} reads
\begin{equation}\label{eq:q11}
-\partial_q \EE(t,q(t)) \in \partial_2 R_\epsilon (\HH(q)(t),\dot{q}(t))\quad \text{in }\hoon  , \quad q(0) = 0,
  \end{equation}a.e.\ in $(0,T)$.
  The viscous dissipation potential   $R_\epsilon:\lo \times  \hon \rightarrow [0,\infty]$ is defined as \begin{equation}\label{def:r11}
R_\epsilon(\zeta,\eta):=R(\zeta,\eta)+\frac{\epsilon}{2}\|\eta\|^2_{\hon},
\end{equation}
where  $\epsilon>0$ is the viscosity parameter and $R$ is the \ro{1-homogeneous} dissipation functional from \eqref{def:r}.

The uniform Lipschitz continuity with respect to\ the viscosity parameter of the solution map to \eqref{eq:q11} plays in the context of optimal control an essential role, see Remark \ref{rem:fw} below. Thanks  to Proposition \ref{lem:ess}, we already have the necessary ingredients to provide this key result in the present paper. \ro{Recall  that Assumption \ref{assu:exd} provides sufficient conditions for the unique solvability of \eqref{eq:q11} (see the beginning of the proof of Theorem \ref{ex1}).}

 \begin{proposition}[Uniform Lipschitz continuity of the viscous solution map]\label{lem:lip_c}
  \ro{Suppose that Assumptions \ref{assu:exd}, \ref{assu:k} and \ref{u''}} 
 hold true. 
Let  $\e>0$ be fixed and let $\ell_1,\ell_2 \in H^1(0,T;\hoon) \cap \ro{\mathfrak{L}(y_0)}$ be given such that   \[\|\ell_i\|_{H^1(0,T;\hoon)} \leq \widehat M, \quad i=1,2\] for some fixed $\widehat M>0$. Then
\begin{equation}\label{visc}
\|\SS_\e(\ell_1)-\SS_\e(\ell_2)\|_{C([0,T];\hon)} \leq c(\widehat M)\,\|\ell_1-\ell_2\|_{\ro{W^{1,1}}(0,T;\hoon)}, 
\end{equation}
where $c(\widehat M)>0$ is independent of $\e$
and  \[\SS_\e:L^2(0,T;\hoon) \to H_0^1(0,T;\hon)\] denotes the solution operator to the viscous problem \eqref{eq:q11}, cf.\,Lemma \ref{lem:proj}.\end{proposition} 

   \begin{proof}
   Throughout the proof, we abbreviate for simplicity $q_i:=\SS_\e(\ell_i), i=1,2.$ According to Lemma \ref{prop:equiv} (applied to the setting from Assumption \ref{assu:c}, see also \eqref{set}), one has 
\begin{equation}\label{eq:evi1}
\e( \dot q_1(t),\eta-\dot q_1(t))_{\hon}  \geq \dual{\ro{-\UU'(q_1(t))} +\ell_1(t)-\F(q_1)(t)}{\eta-\dot q_1(t)}_{\hon} \  \forall\, \eta \in \CC, 
\end{equation}
\begin{equation}\label{eq:evi11}
\e( \dot q_2(t),\eta-\dot q_2(t))_{\hon}  \geq \dual{\ro{-\UU'(q_2(t))} +\ell_2(t)-\F(q_2)(t)}{\eta-\dot q_2(t)}_{\hon} \ \forall\, \eta \in \CC,
\end{equation}\ae $(0,T)$.
Testing \eqref{eq:evi1} with $\dot q_2$ and \eqref{eq:evi11} with $\dot q_1$ and adding the resulting VIs yields 
\begin{equation*}\begin{aligned}
&\e( \dot q_1(t)-\dot q_2(t),\dot q_2(t)-\dot q_1(t))_{\hon}
\\&\geq \ro{\dual{\UU'( q_2(t))-\UU'( q_1(t))}{\dot q_2(t)-\dot q_1(t)}_{\hon}}
-\dual{\ell_1(t)-\ell_2(t)}{\dot q_1(t)-\dot q_2(t)}_{\hon}
\\&\quad+\dual{\F(q_2)(t)-\F(q_1)(t)}{\dot q_2(t)-\dot q_1(t)}_{\hon} \quad \ae (0,T).
\end{aligned}\end{equation*}\ro{Thus, for all $t\in [0,T]$ it holds 
\begin{equation}\label{sigma_t0}\begin{aligned}
 \int_0^t \tilde \sigma(s)\,ds
&  \leq -\int_0^t \dual{\dot \ell_1(s)-\dot \ell_2(s)}{ q_1(s)- q_2(s)}_{\hon} \,ds
\\&\quad +\dual{\ell_1(t)-\ell_2(t)}{ q_1(t)- q_2(t)}_{\hon}+\int_0^t E(s)\,ds,
 \end{aligned}\end{equation} where we abbreviate
 \[\tilde \sigma (s):= \ro{\dual{\UU'( q_2(s))-\UU'( q_1(s))}{\dot q_2(s)-\dot q_1(s)}_{\hon}} ,\]
 \[E(s):=-\dual{\F(q_2)(s)-\F(q_1)(s)}{\dot q_2(s)-\dot q_1(s)}_{\hon}, \quad s\in[0,t].\]
From Proposition \ref{lem:est} we infer that $\|q_i\|_{H^1(0,T;{\hon})} \leq C\,\|\ell_i\|_{H^1(0,T;\hoon)}{+c_0},i=1,2$, where $C,c_0>0$ are independent of $\e$. In light of Assumption \ref{assu:k}, this allows us to  apply 
  Proposition \ref{lem:ess} with \[M:=C\,\widehat M{+c_0}.\] Note that $M>0$ is independent of $\e$. By arguing as in the proof of \eqref{sigma_tt}, we see that \eqref{sigma_t0} can be continued as
  \begin{equation}\label{sigma_tt0}
  \begin{aligned}
   \int_0^t \tilde \sigma(s)\,ds &\leq 
 C\| \ell_1-\ell_2\|_{W^{1,1}(0,T;\hoon)} \|q_2- q_1\|_{C([0,T];\hon)}
 \\&\qquad +  \frac{\alpha}{4}\|(q_2- q_1)(t)\|_{\lo}^2
 \\&\qquad +C(L_\kappa,\alpha,T, \widehat M, {\varrho},\O) \int_0^t \| (q_2- q_1)(s)\|_{\hon}^2\,ds  \quad \forall\,t\in [0,T],
  \end{aligned}\end{equation} where all the constants appearing in \eqref{sigma_tt0} depend only on the given data.
Then, arguing in the exact same manner as in the proof of Theorem \ref{thm:lip} gives the desired assertion.
}
 \end{proof}
 \begin{remark}
In  \cite[Lem.\,4.5]{sww}, an  estimate similar to \eqref{visc} is provided for the smoothened viscous solution map. Note that, as opposed to Proposition \ref{lem:lip_c}, the result in \cite[Lem.\,4.5]{sww} does not require  uniform boundedness and  additional regularity of the controls. This is due to the fact that the dissipation functional  in \cite{sww} is \ro{state-independent} (a history term is not taken into consideration).
\end{remark}
 \begin{remark}\label{rem:fw}
The estimate in Proposition \ref{lem:lip_c} is fundamental for the derivation of optimality conditions for the control of \eqref{eq:min2}. This result has been recently employed in {\cite{bp}}, where an optimality system for the \ro{finite-dimensional} version of \eqref{eq:min2} is established. Therein,  the \ro{finite-dimensional} counterpart of Proposition \ref{lem:lip_c} led to  suitable uniform estimates with respect to\ $\e$ for the (viscous) multipliers appearing in the optimality conditions \ro{associated with }the control of \eqref{eq:q11} (see \cite[proof of Prop.\,3.6]{bp}). This allowed us to  perform a vanishing viscosity analysis  and obtain a strong limit optimality system. 

In a follow-up paper, our aim is  to extend these findings  to the \ro{infinite-dimensional} case. \ro{Let us stress that the strongest kind of optimality conditions for the control of the viscous regularization  \eqref{eq:q11} have already been developed in \cite[Sec.\,4.1]{aos}. This shows that the "viscous" optimality conditions are a good point of departure when it comes to obtaining a satisfying optimality system for the control of \eqref{eq:n1}.}
\end{remark}

\begin{appendix}
\section{Fenchel conjugate functionals}\label{a}
In this section, we compute the conjugate functionals of some of the mappings mentioned in Remark \ref{rem:r}. 

Since, for a given $\zeta \in X$,  $\RR(\zeta,\cdot)$ is proper, convex and lower semicontinuous, cf.\,Assumption \ref{it:st1}.\ref{it:st11}, we have the equivalence 
\[y^\ast \in \partial_2 \RR(\zeta,y)  \quad  \text{in }Y^\ast\Leftrightarrow y \in \partial [\RR(\zeta,\cdot)]^\star (y^\ast ) \quad \text{in }Y.\] By making use thereof, we see that the respective assertions in Remark \ref{rem:r} are true.
\begin{lemma}\label{conj0}
Let $g:\lo \to \lo$ be a given mapping. 
For $\zeta \in L^2(\O)$, the Fenchel conjugate of the functional $\widehat R(\zeta,\cdot): H^1(\O) \to \R$ defined as
\[\widehat R(\zeta,\eta)=\int_{\O} g(\zeta) | \eta|\,dx,\] cf.\,\eqref{def:t0}, is given by the mapping  $[\widehat R(\zeta,\cdot)]^\star:H^1(\O)^\ast \to [0,\infty]$,  
\[[\widehat R(\zeta,\cdot)]^\star(\omega)=\II_{K(\zeta)}(\omega) \quad \forall\,\omega \in H^1(\O)^\ast,\]
where
\begin{equation*}\begin{aligned}
K(\zeta): =\{\phi \in L^2(\O):|\phi|\leq g (\zeta)\  \ae \O\}.\end{aligned}\end{equation*}
\end{lemma}
\begin{proof}Let $\omega \in \hoon$ be arbitrary but fixed. 
In view of the definition of the Fenchel conjugate it holds 
\begin{equation*}\begin{aligned}
[\widehat R(\zeta,\cdot)]^\star(\omega)=\sup_{v \in \hon} \, \Big(\dual{\omega}{v}_{\hon}- \int_{\O} g(\zeta) |v|\,dx\Big).
\end{aligned}\end{equation*}
If there exists $\widehat v \in \hon$ \ro{such} that $f(\widehat v):=\dual{\omega}{\widehat v}_{\hon}-\int_{\O} g(\zeta) | \widehat v|\,dx>0$, then 
\begin{equation*}\begin{aligned}
[\widehat R(\zeta,\cdot)]^\star(\omega)\geq \sup_{\gamma \in \R^+} \gamma f(\widehat v)=\infty.\end{aligned}\end{equation*}
If  $\dual{\omega}{ v}_{\hon} \leq \int_{\O} g(\zeta) | v|\,dx$ for all $v \in \hon$, then 
\[ [\widehat R(\zeta,\cdot)]^\star(\omega)= 0.\]
Since, by Hahn-Banach's theorem and $g(\zeta) \in \lo$, one has 
\begin{equation*}\begin{aligned}
K(\zeta)
=\{\phi \in \hon^\ast:\dual{\phi}{v}_{\hon}\leq \int_\O g (\zeta) |v |\,dx \quad \forall\,v \in \hon \},\end{aligned}\end{equation*} 
 the proof is now complete.
\end{proof}

\begin{lemma}\label{conj}
Let $\zeta \in \lo$. 
The Fenchel conjugate of $R(\zeta,\cdot): \hon \to [0,\infty]$, see \eqref{def:r0} and \eqref{def:r}, is given by the mapping  $[R(\zeta,\cdot)]^\star:\hoon \to [0,\infty]$,  
\[[R(\zeta,\cdot)]^\star(\omega)=\II_{\CC^\circ}(\omega-\kappa(\zeta)) \quad \forall\,\omega \in \hoon,\]where $\CC$ is the set defined in \eqref{def:c}.
\end{lemma}
\begin{proof}
In view of the definition of the Fenchel conjugate and \eqref{def:r} it holds 
\begin{equation*}\begin{aligned}
[R(\zeta,\cdot)]^\star(\omega)&=\sup_{v \in \CC} \, \dual{\omega}{v}_{\hon}- \int_\Omega \kappa(\zeta) \,  v \;dx
\\&=\sup_{v \in \CC} \, \dual{\omega-\kappa(\zeta)}{ v}_{\hon}
\\&=(\II_{\CC})^\star(\omega-\kappa(\zeta)) .
\end{aligned}\end{equation*}Since $\CC$ is a convex cone, the desired result can now be concluded, as $(\II_{\CC})^\star=\II_{\CC^\circ}$ \cite{rockaf}.
\end{proof}

\end{appendix}
\section*{Acknowledgment}This work was supported by the German Research Foundation under the DFG grant BE 7178/3-1 for the project ``Optimal Control of Viscous
Fatigue Damage Models for Brittle Materials: Optimality Systems". {The author is thankful to Gerd Wachsmuth (BTU Cottbus-Senftenberg) for pointing out an inaccuracy in the proof of Proposition \ref{lem:ess} in a previous version of the manuscript.}

\end{document}